\documentclass[]{article}
\usepackage{graphicx}
\usepackage{caption}
\usepackage{subcaption}
\usepackage{algorithmic}
\usepackage{algorithm}
\usepackage{amsmath}
\usepackage{amssymb}
\usepackage{amsfonts}
\usepackage{amsthm}
\usepackage{fullpage}
\usepackage{url}
\usepackage{color}
\usepackage{authblk} 
\usepackage{float}
\usepackage{comment}

\def\N{{\mathbb{N}}}
\def\R{{\mathbb{R}}}

\def\IMF{{\textrm{IMF}}}

\newtheorem{theorem}{Theorem}

\newtheorem{corollary}{Corollary}
\newtheorem{definition}{Definition}

\begin{document}

\title{New theoretical insights in the decomposition and time-frequency representation of nonstationary signals: the IMFogram algorithm.}
\author{Antonio Cicone\thanks{DISIM, Universit\`a degli Studi dell'Aquila, L'Aquila, ITALY, and Istituto di Astrofisica e Planetologia Spaziali, INAF, Rome, ITALY, and Istituto Nazionale di Geofisica e Vulcanologia, Rome, ITALY ({\tt antonio.cicone@univaq.it})}
\and
Wing Suet Li\thanks{School of Mathematics, Georgia Institute of Technology, Atlanta, GA, USA. ({\tt li@math.gatech.edu})}
\and
Haomin Zhou\thanks{School of Mathematics, Georgia Institute of Technology, Atlanta, GA, USA. ({\tt hmzhou@math.gatech.edu})}
}

\maketitle
\date{}

\begin{abstract}

The analysis of the time--frequency content of a signal is a classical problem in signal processing, with a broad number of applications in real life. Many different approaches have been developed over the decades, which provide alternative time--frequency representations of a signal each with its advantages and limitations. In this work, following the success of nonlinear methods for the decomposition of signals into intrinsic mode functions (IMFs), we first provide more theoretical insights into the so--called Iterative Filtering decomposition algorithm, proving an energy conservation result for the derived decompositions. Furthermore, we present a new time--frequency representation method based on the IMF decomposition of a signal, which is called IMFogram. We prove theoretical results regarding this method, including its convergence to the spectrogram representation for a certain class of signals, and we present a few examples of applications, comparing results with some of the most well know approaches available in the literature.

\end{abstract}

\section{Introduction}

The study of techniques for the time--frequency (TF) analysis of signals is a long lasting line of research in signal processing which have led over the decades to the development of many new algorithms and approaches, which are nowadays commonly used practically in many field of research \cite{flandrin1998time}. One of the main goals of these techniques is the derivation of what is called the TF representation (TFR) of a given signal. The TFR is a two dimensional interpretation of a monodimensional signal, where one dimension relates to time and the other one to frequency. Studying the TFR of a signal allows in many cases to unveil its hidden features, especially the nonstationary ones.

There exist many types of TF analysis algorithms. Most of them belong to either linear or bilinear methods \cite{flandrin1998time}. In linear methods, the signal is studied via inner products with (or correlations with) a pre-assigned basis. The most well known methods are the windowed Fourier transform \cite{cohen1995time}, and the wavelet transform \cite{daubechies1992ten}. However, in all these methods, and related ones, the chosen basis unavoidably leaves its peculiar footprint in the derived TFR, which may badly affect its interpretation when looking for properties of the signal. Moreover, the Heisenberg uncertainty principle limits the maximal achievable resolution of the TF plane due to blurring or smearing out of the TFR. Different choices of the linear transform or basis can produce different trade-offs, but none is ideal \cite{daubechies2011synchrosqueezed,wu2020current}.
In the bilinear methods for TFR, like the Wigner-Ville distribution and, more generally, the Cohen's class or affine class, one can avoid introducing a basis with which the signal is compared or measured. As a consequence, some features can have a crisper and more focused representation in the TF plane with these methods. However, as a side effect, reading the TFR of a multi-component signal becomes more complicated by the presence of interference terms between the TFRs of the single components. These interferences lead in most cases the TF density to be even negative in some parts of the TF plane. This last flaw can be removed by some postprocessing of the representation, which, however, reintroduce unavoidably some blur in the TF plane \cite{flandrin1998time}. Furthermore, the extraction of a signal, or part of it, is much less straightforward for bilinear than for linear TFRs. To make things worse, in many practical applications the signals under study contain several components which are highly nonstationary, with time varying characteristics which may be important to capture as accurately as possible. For such signals both the linear and bilinear methods show their limitations. Bilinear methods TFRs end up being corrupted by many interference terms, and even if these interferences can be mitigated, the reconstruction of the individual components would still be a hard problem to solve. Linear methods are too rigid, or provide too blurred TFRs \cite{daubechies2011synchrosqueezed}.

For all these reasons in the last few years several approaches have been proposed, like the reassignment method \cite{auger1995improving,chassande2003time}, the scattering transform and the time-frequency scattering \cite{mallat2012group,anden2019joint}, the synchrosqueezing transform and the concentration of
frequency and time \cite{daubechies1996nonlinear,daubechies2011synchrosqueezed,auger2013time,daubechies2016conceft}. For a more comprehensive list and further details on these methods please refer to \cite{wu2020current}. All these methods allow to produce a much crisper and focused TF representation of a highly nonstationary multi-components signal. However the problem of how to properly extract chirps from a given signal remains an open problem even with these approaches. This is because they have been developed for TF analysis and they have only limited abilities in the decomposition and separation of components, especially when the noise level increases \cite{daubechies2016conceft}. New promising results in this direction are coming from the so called de-shape algorithm and Ramanujan de-shape method \cite{lin2018wave,chen2020ramanujan}. Furthermore, the extension of these kind of methods to higher dimensional and multivariate signals has been only partially explored so far \cite{lin2019wave}.

Two decades ago a different kind of method was introduced, the so called Hilbert Huang Transform (HHT) \cite{huang1998empirical}, a local and adaptive data-driven method which has an iterative ``divide et impera'' approach. The idea is simple, but powerful: we first iteratively divide the signal into several simple oscillatory components via the so called Empirical Mode Decomposition (EMD) method, then each of them is analyzed separately in the TF domain via the well known Hilbert transform, i.e. the computation of the instantaneous frequency of each component \cite{huang2014introduction,huang2009instantaneous}. This approach allows to bypass in part the Heisenberg-Gabor uncertainty principle \cite{flandrin1998time}, overcoming, in particular, artificial spectrum spread caused by sudden changes. EMD allows to produce each simple oscillatory component via the subtraction of the signal moving average which is computed as the average between two envelopes connecting its minima and maxima \cite{huang1998empirical}. These simple oscillatory components have been named intrinsic mode functions (IMFs) by Huang and collaborators \cite{huang1998empirical} and are informally defined in  that work as functions fulfilling two properties:
the number of extrema and the number of zero crossings must either equal or differ at most by one;
the mean between an upper envelope, connecting all the local maxima, and a lower envelope, connecting all the local minima of the function, has to be zero at any point.

The decompositions produced using the EMD algorithm proved to be successful for a wide range of applications, like, for example, \cite{cummings2004travelling,parey2006dynamic,yu2008forecasting,huang2008review,tary2014spectral,stallone2020new}. Nevertheless, the EMD algorithm contains a number of heuristic and ad hoc elements that make hard its mathematical analysis, including an a priori convergence. This is because the core of the algorithm relies heavily on interpolates of signal maxima and minima. This very approach does have also some stability problems in the presence of noise, as illustrated in \cite{huang2009EEMD}. Several variants of the EMD have been recently proposed to address this last problem, e.g. the Ensemble Empirical Mode Decomposition (EEMD) \cite{huang2009EEMD}, the complementary EEMD \cite{yeh2010complementary}, the complete EEMD \cite{torres2011complete}, the partly EEMD \cite{zheng2014partly}, the noise assisted multivariate EMD (NA-MEMD) \cite{rehman2009Filterbank}. They all allow to address this issue as well as reduce the so called mode mixing problem. But they pose new challenges both to our mathematical understanding of this kind of techniques and to the ability of these methods to handle chirps, since they worsen the mode-splitting problem present in the EMD algorithm \cite{yeh2010complementary}.
Given the attention that these methods received from the worldwide scientific community (Huang papers received so far more than 30000 citations based on Scopus) many other research groups started working on this topic and proposed their alternative approaches to signals decomposition. We recall the sparse TF representation \cite{hou2011adaptive,hou2009variant}, the Geometric mode decomposition \cite{yu2018geometric}, the Blaschke decomposition \cite{coifman2017carrier}, the Empirical wavelet transform \cite{gilles2013empirical}, the Variational mode decomposition \cite{dragomiretskiy2013variational}, and similar techniques \cite{selesnick2011resonance,meignen2007new,pustelnik2012multicomponent}. All of these methods are based on optimization with respect to an a priori chosen basis.

The only alternative method proposed so far in the literature which is based on iterations, and hence does not require any a priori assumption on the signal under analysis, is Iterative Filtering (IF) algorithm \cite{lin2009iterative}, its fast implementation based on FFT, named Fast Iterative Filtering (FIF) \cite{cicone2020numerical}, and their generalizations, the Adaptive Local Iterative Filtering (ALIF) and Resampled Iterative Filtering (RIF) algorithms \cite{barbarino2020conjectures,barbarino2021stabilization} for the handling of signals containing strongly nonstationarities, like chirps, whistles and multipaths. These alternative iterative methods, although published only recently, have already been used effectively in a wide variety of applied fields like, for instance, in \cite{yu2010modeling,sharma2017automatic,li2018entropy,mitiche2018classification,spogli2019role,materassi2019stepping,papini2020multidimensional,piersanti2020magnetospheric,spogli2020adaptive,piersanti2020inquiry}. The structure of the IF, ALIF and RIF algorithms resemble the EMD one. Their key difference is in the way the signal moving average is computed, i.e., via correlation of the signal itself with an a priori chosen filter function, instead of using the average between two envelopes. This apparently simple difference opened the doors to the mathematical analysis of IF, ALIF and RIF \cite{huang2009convergence,cicone2020study,cicone2020numerical,barbarino2020conjectures,stallone2020new,cicone2022oneortwo}.

In this work we tackle the problem of building a fast, local and reliable TFR method based on the IMFs decomposition produced by the IF--based methods. To do so, we present some new insight in the study of the IF--based methods, and we analyze in details the properties of the derived IMFs. Based on these results we propose a new TFR technique, named IMFogram (pronounced like ``infogram''), which proves to be a generalization of the spectrogram of a signal \cite{flanagan2013speech}. We point out that the IMFogram method requires only the IMFs decomposition to produce in a fast, local and reliable way the TFR of a signal. Therefore this method is not limited only to IF--based IMFs decompositions. IMFogram can be used to produced a TFR based on the IMFs derived by any decomposition method, like the EMD--based techniques, or any other signal decomposition method developed so far.

The rest of the work is organized as follows. In Section \ref{sec:IF} we review the IF method, its fast implementation based on FFT, named Fast Iterative Filtering (FIF), and their main theoretical properties. Section  \ref{sec:IMFogram} is devoted to presenting new theoretical insight in the decompositions produced by IF--based methods. Then, leveraging on these new theoretical insights, we present the IMFogram, and we study its mathematical properties. In Section \ref{sec:NumericalExamples} we show some numerical examples of applications of the newly proposed IMFogram to both artificial and real life signals, and the comparisons with other 
methods, like spectrogram, a.k.a. Short Time Fourier Transform, synchrosqueezing, and Hilbert transform based TFRs. This work ends with an outlook to future research directions.

\section{(Fast) Iterative Filtering}\label{sec:IF}

To begin, we recall briefly the Fast Iterative Filtering (FIF) method  that decomposes a discrete signal $s$ into a collection of IMFs. To keep the discussion concise, our signal is always a finite sequence $s=(s_j)_{0}^{n-1}$ in $\mathbb{R}^n$. This finite sequence should be viewed as a time series over the time interval $[0,L]$, evenly sampled at the rate of $2B$. More precisely, for $j=0, 1, \dots,  n-1$, $n=2BL$,
\begin{equation}
    t_j =\frac{j}{2B},\quad s_j = s(t_j).\label{eq:t_j}
\end{equation}

Thus, $(s_j)$ can be viewed as a discretization of $s$. With a bit of abuse of notation, $s$ refers to both the underlining signal and the discretization of the signal. From a discrete signal $s=(s_j)$ in $\mathbb{R}^n$, we can also view the signal $s$ is a continuous function by a linear interpolation of $(s_j)$ over the interval $[0,L]$.

For convenience, rather than use the usual $2$-norm of a vector in $a=(a_k)_{1\le k\le n}\in \mathbb{R}^n$ which corresponds to a signal over a time domain [0,L], as $\displaystyle{\left( \sum_k |a_k|^2\right )^{1/2}}$, throughout the paper, we use the normalized $2$-norm:
\[
\|a\|_2=\frac{L}{n}\left( \sum_j |a(t_j)|^2\right )^{1/2}
\]
where the function $a: [0,L]\to \mathbb{R}$ is obtained as the linear interpolation of the sequence $(a_k)$ over the interval $[0,L]$.

The discrete Fourier transform of $(s_j)$, $\hat{s}_k=\hat{s}(\xi_k)$, where
\[
\xi_k=\frac{k}{L}, \quad k=0,1,\dots,n-1\label{eq:xi_k}
\] and
\[
\hat{s}(\xi_k)=\sum_{j=0}^{n-1}s_je^{-2{\pi}is_j{\xi}_k}.
\]
So the sampling rate for $\hat{s}$, the discrete Fourier transform of $s$ is $\frac{1}{L}$ over the frequency interval $[0,2B]$.

For all the signals, and later any filters, that we will consider here, we assume that they are all pre-extended to eliminate any boundary error propagation that may occur due to the usage of Fast Fourier transformation (FFT). For a detailed discussion, please see \cite{cicone2020study,stallone2020new}.


A key object in the iterative filtering method is the filter. We recall the definition here for completeness.

\begin{definition}\label{def:filter}
\begin{itemize}
    \item [(1)] A function $w:[-l, l]\to \mathbb{R}$ is a filter if it is nonnegative, even, bounded, continuous, and $\int_{\mathbb{R}} w(t)\,dt = 1$.
    \item[(2)] A double convolution filter $w$ is the self-convolution of a filter $\tilde{w}$, that is $w=\tilde{w}*\tilde{w}$.
    \item[(3)] The size, or the length, of a filter $w$ is half of its support: filter length of $w= \frac{1}{2}\,m\{t: w(t)>0\}$.
\end{itemize}
\end{definition}

The definition implies that convolutions of filters are filters. It is easy to see that the range of the Fourier transform of a filter is in $[-1, 1]$, and the range of the Fourier transform of a double convolution filter is in $[0,1]$.

We will also identify the filter $w$ with its discretization: $w=(w_j) = (w(t_j))$, where $t_j$ defined as in \eqref{eq:t_j}.
Consider a signal $s=(s_j)$ and a double convolution filter $w=(w_j)$ in $\mathbb{R}^n$. Both sequences $(s_j)$ and $(w_j)$ are properly extended, as mentioned earlier, so that there will be no boundary error when FFT is applied.

The moving average with respect to the filter $w$ is defined as the convolution
\begin{equation}
    \mathcal{C}_ws=w*s,\quad (\mathcal{C}_ws)_j = \sum_{k=0}^{n-1}w_{k+j}s_k.\label{eq:moving average}
\end{equation}
We subtract the moving average from the signal and obtain the variation of the signal around its $w$-moving average,
\begin{equation}
    \mathcal{V}_ws=s-\mathcal{C}_w*s,\quad (\mathcal{V}_ws)_j=s_j-(\mathcal{C}_w*s)_j.\label{eq:variation of signal}
\end{equation}
Iterating $p$-times the linear operator $\mathcal{V}_w$, we obtain the linear IMF operator,
\begin{equation}
    \mathcal{I}_{w,p}=\mathcal{V}_w^p \label{eq:IMFoperator}
\end{equation}

Observe that in the frequency domain,
\begin{equation}
    \left(\widehat{\mathcal{V}_ws}\right)_j = (1-\widehat{w})_j (\widehat{s})_j, \quad
    (\widehat{\mathcal{I}_{w,p}s})_j=(1-(\widehat{w})_j)^p(\widehat{s})_j. \label{eq:IMFoperatorFrequencyDom}
\end{equation}
Taking the inverse FFT (iFFT), we obtain the an IMF from $s$ with filter $w$ and $p$ iterations:
\begin{equation}
 \textrm{IMF} = \textrm{IMF}_{w,p} =  \textrm{iFFT}(\widehat{\mathcal{I}_{w,p}s}),  \label{eq:IMFoperator}
\end{equation}

We recall the following theorem from \cite{cicone2020numerical}, which guarantees that the difference of consecutive iterations of the variation operators converges to zero as the number of iterations increases.

\begin{theorem}[\cite{cicone2020numerical}]\label{thm:IF_inner_conv_stopping}
    Let $s=(s)_j$ be a discrete signal with a double convolution filter $w=(w_j)$.  Then for all $p$,
\begin{equation}\label{eq:IMF_IF_stop}
\left\| \mathcal{V}_w^{p+1}(s)-\mathcal{V}_w^{p}(s)\right\|_2\le\frac{1}{ep}\|s\|_2.
\end{equation}
\end{theorem}

As we see, an IMF depend on choices of the filter $w$ and the number of iterations $p$.
Theorem \ref{thm:IF_inner_conv_stopping} provides the theoretical justification for the stopping criterion for the inner loop of the following FIF algorithm. In practice, number of iterations for the inner loop is much smaller than the theoretical bound.

\begin{algorithm}
\caption{\textbf{Fast Iterative Filtering} IMF = FIF$(s)$,  $\delta >0$}\label{algo:FIF}
\begin{algorithmic}
\STATE IMF = $\left\{\right\}$
\WHILE{the number of extrema of $s$ $\geq 2$}
      \STATE  compute the filter length $l$ for $s$ and the corresponding filter $w$
      \STATE  $\widehat{s}=\textrm{FFT}(s)$
      \STATE  $\widehat{w}=\textrm{FFT}(w)$
    \WHILE{$\|\widehat{s_{m+1}}-\widehat{s_m}\|_2>\delta\|s\|_2$}
                  \STATE  $(\widehat{s_{m+1}})_j = (\widehat{\mathcal{I}_{w,m}s})_j$, for $j=0,...,n-1$

                  \STATE  $m = m+1$
      \ENDWHILE
      \STATE IMF = IMF$\,\cup\,  \{ \textrm{iFFT}\left(\widehat{s}_{m}\right)\}$
      \STATE $s=s-\textrm{iFFT}\left(\widehat{s}_{m}\right)$
\ENDWHILE
\STATE IMF = IMF$\,\cup\,  \{ s\}$
\end{algorithmic}
\end{algorithm}

Regarding the filter, we choose a double convolution filter $w_0$ like the Fokker-Plank filters \cite{cicone2016adaptive}, and for the length $l$ estimation, following \cite{lin2009iterative}, we compute
\begin{equation}\label{eq:Unif_Mask_length}
l:=2\left\lfloor\chi \frac{n}{k}\right\rfloor
\end{equation}
where $n$ is the length of the vector $s$, $k$ is the number of its extrema, $\chi$ is a tuning parameter choose between 1.1 and 2, and usually fixed to 1.6, and $\left\lfloor \cdot \right\rfloor$ rounds a positive number to the nearest integer closer to zero. We point out that this formula computes some sort of average highest frequency contained in $s$.

Another possible approach could be using the Fourier spectrum of $s$ and the identification of its highest frequency peak. The filter length $l$ can be chosen to be proportional to the reciprocal of this value.

The Fokker-Plank filters, as well the approaches just described for the computation of the filter length $l$ are pretty standard, we welcome the readers to design different  filters.

The computation of the filter length $l$ is a crucial step of the FIF technique \cite{cicone2022oneortwo}. Clearly, $l$ is strictly positive and, more importantly, it is based solely on the signal itself. This last property makes the method nonlinear.

In fact, if we consider two distinct signals $s$ and $z$ of (the same length over the same time interval with the same sampling rate),  then in general there is no relationship between the filter length associated with $s$, $z$, and $s+z$. Hence, if
$\IMF_1(\bullet)$ represents the first IMF extracted by FIF from a signal, then

$$\IMF_1(s+z)\neq \IMF_1(s)+\IMF_1(z).$$

\section{New theoretical insights in nonstationary signals decomposition and their TFR}\label{sec:IMFogram}

In this section we present new theoretical insights in the decompositions produced using FIF--based methods and on the time-frequency representation for nonstationary signal based on the IMFs. Using these results, we propose a new tool to study nonstationary signals named IMFogram, which was first mentioned in \cite{Barbe2022time}, but was not discussed in detail.

It has been well established that the $2$-norm represent the energy for a signal, either discrete or continuous. It follows from the Parseval-Plancherel identity that Fourier transform is an isometry, therefore the $2$-norm of an $L_2$ signal is the same as the $2$-norm of its Fourier transform. Here we propose a new energy norm for a signal, the $L_1$ Fourier Energy, when it is finite.

\begin{definition}[$L_1$ Fourier Energy of a signal $s$]\label{def:L1_Energy}
Given a signal $s$, its $L_1$ \textit{Fourier Energy} is defined as
\begin{equation}\label{eq:L1_Energy}
E_{1}(s)=||\widehat{s}||_1.
\end{equation}
In particular, if $s=(s_j)\in\mathbb{R}^n$, then $E_1(s) =
\displaystyle{\sum_k |\widehat{s}(\xi_k)|}$ where $\xi_k$ is defined as in (\ref{eq:xi_k}).
\end{definition}

\begin{definition}[Conservation of $L_1$ Fourier Energy]\label{def:L1_Energy}
Let $s\in\mathbb{R}^n$ and $s=\sum_k \phi_k$ be a decomposition of $s$. We say that the decomposition conserves the signal $L_1$ Fourier Energy if and only if
\begin{equation}\label{eq:Conservation_L1_Energy}
E_{1}(s)=\sum_k E_{1}(\phi_k)
\end{equation}
\end{definition}

\begin{definition}[Unwanted oscillations] \label{def:Unwanted_oscillations}
Let $s\in\mathbb{R}^n$ and $s=\sum_k \phi_k$ be a decomposition of $s$. The decomposition $\{\phi_k\}$ contains unwanted oscillations if there exists $\xi$ such that
\begin{equation}\label{eq:Unwanted_oscillations}
    \sum_k \left|\widehat{\phi}_k(\xi)\right| > \left|\widehat{s}(\xi)\right|
\end{equation}
\end{definition}

Our new theoretical insight is that Algorithm \ref{algo:FIF} will not produce any unwanted oscillation and the decomposition preserves the $L_1$ Fourier Energy norm.

\begin{theorem}\label{thm:Energy_conservation}
Let $s\in \mathbb{R}^n$ and $\delta >0$. Apply Algorithm \ref{algo:FIF} with $w$ a double convolution filter, and we have $s=\displaystyle{\sum_1^m \textrm{IMF}_k + r}$, where $r$ is a trend. Then this decomposition preserves the $L_1$ Fourier energy and produces no unwanted oscillations.
\end{theorem}

\begin{proof}
Following Algorithm \ref{algo:FIF}, let $s^{(1)}=s$, and $w_1$ be the associated double convolution filter. Then
\begin{equation*}
    \IMF_1 = \mathcal{I}_{w_1,p_1}(s^{(1)}) \label{eq:IMF_1}
\end{equation*}
for some positive integer $p_1$ determined by $\delta$,
and
\begin{equation*}
    \widehat{\IMF}_1(\xi) = (1-\widehat{w}_1(\xi))^{p_1} \widehat{s^{(1)}}(\xi).
\end{equation*}
Inductively, for $k = 1, 2, ..., m-1$, define
\begin{equation*}
    s^{(k+1)} = s^{(k)} - \IMF_k, \quad
    \widehat{\IMF}_{k+1}(\xi) = (1-\widehat{w_{k+1}}(\xi))^{p_{k+1}} \widehat{s^{(k+1)}}(\xi),
\end{equation*}
where $w_k$ is the double convolution filter associated with $s^{(k)}$ and $p_k$ is associated with the stopping criterion $\delta$. It follows immediately that for $k=1, 2, \dots, m$
\begin{equation*}
    \widehat{\IMF}_{k}(\xi) =\widehat{s}(\xi) (1-\widehat{w_{k}}(\xi))^{p_{k}}\prod_{j=1}^{k-1}\left(1 - (1-\widehat{w}_j(\xi))^{p_j}\right).
\end{equation*}
Define $f_k(\xi)= (1-\widehat{w_{k}}(\xi))^{p_{k}}\prod_{j=1}^{k-1}\left(1 - (1-\widehat{w}_j(\xi))^{p_j}\right) $ and $f_0(\xi) = \widehat{r}(\xi)$. Observe that
\begin{equation*}
    \sum_{k=0}^m f_k(\xi) = 1, \quad 0\le \sum_{k=1}^m f_k(\xi) \le 1, \quad \textrm{for all sampling points } \xi,
\end{equation*}
therefore, $0\le f_0(\xi) = \widehat{r}(\xi)\le 1$. It follows immediately that
\begin{equation*}
    \sum_k \left|\widehat{\IMF}_k(\xi)\right| = \sum_{k=1}^m f_k(\xi)  \left|\widehat{s}(\xi)\right|\le \left|\widehat{s}(\xi)\right|.
\end{equation*}
The conservation of the $L_1$ Fourier Energy follows from
\begin{align*}
\sum_{k=0}^m ||\widehat{\IMF}_k||_1 &=\sum_{k=0}^m\sum_j f_k(\xi_j)\left|\widehat{s}(\xi_j)\right| = \sum_j \sum_{k=0}^m
f_k(\xi_j)\left|\widehat{s}(\xi_j)\right| \\
&=\sum_j \left|\widehat{s}(\xi_j)\right| = ||\widehat{s}||_1
\end{align*}
\end{proof}

An important result that follows from Theorem \ref{thm:Energy_conservation} is that, intuitively speaking, the Fourier transforms of the first IMF contains, for every fixed frequency $\xi$, a percentage, precisely $f_1(\xi)=(1-\widehat{w}_1(\xi))^{p_1}$, of $\widehat{s}(\xi)$, i.e. the original signal Fourier transforms at frequency $\xi$. Subsequent IMFs contains at $\xi$ a percentage of the leftover signal Fourier transforms in $\xi$. For instance the second IMF Fourier transform contains at $\xi$ the percentage $f_2(\xi)=(1-\widehat{w}_2(\xi))^{p_2}\left(1 - (1-\widehat{w}_1(\xi))^{p_1}\right)$, which is the $(1-\widehat{w}_2(\xi))^{p_2}$ of the percentage $\left(1 - (1-\widehat{w}_1(\xi))^{p_1}\right)$ leftover in the signal after subtracting the first IMF. The very same reasoning applies to all the subsequent IMFs. This observation allows to state in an intuitive way that, when we add together all the IMFs from the first to the $m$-th one, the
$\sum_{k=1}^m \widehat{\IMF}_k(\xi)$ is given by the complex number $\widehat{s}(\xi)$ multiplied by a real number in the interval $[0,\ 1]$. Furthermore, this number is strictly monotonically increasing and tends to 1 as $m\in\N$ grows.

Furthermore, this theorem justify our definition of $L_1$ Fourier Energy, energy conservation, and of what it can be considered as an unwanted oscillation potentially produced by FIF methods.  This implies also that the conservation of $L_1$ Fourier Energy in FIF--based methods is all we need  to guarantee the meaningfulness of the decomposition.

\subsubsection{Example of unwanted oscillations: case of FIF with a bad choice for the filter}

If we run FIF with a filter which is not double convolution filter we may end up introducing unwanted oscillations in the decomposition.

\begin{figure}
\centering
\includegraphics[width=0.29\linewidth]{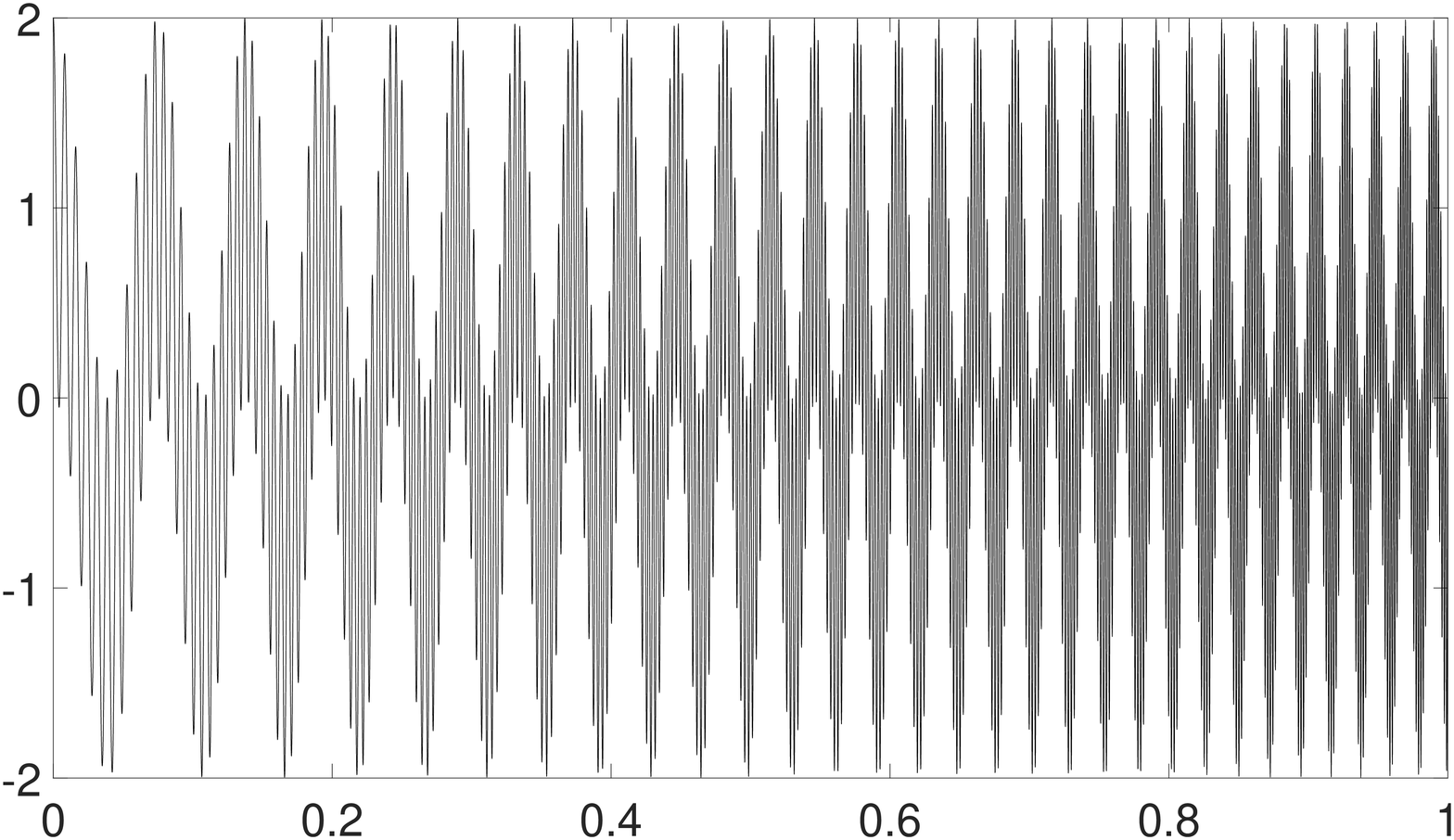}~\includegraphics[width=0.34\linewidth]{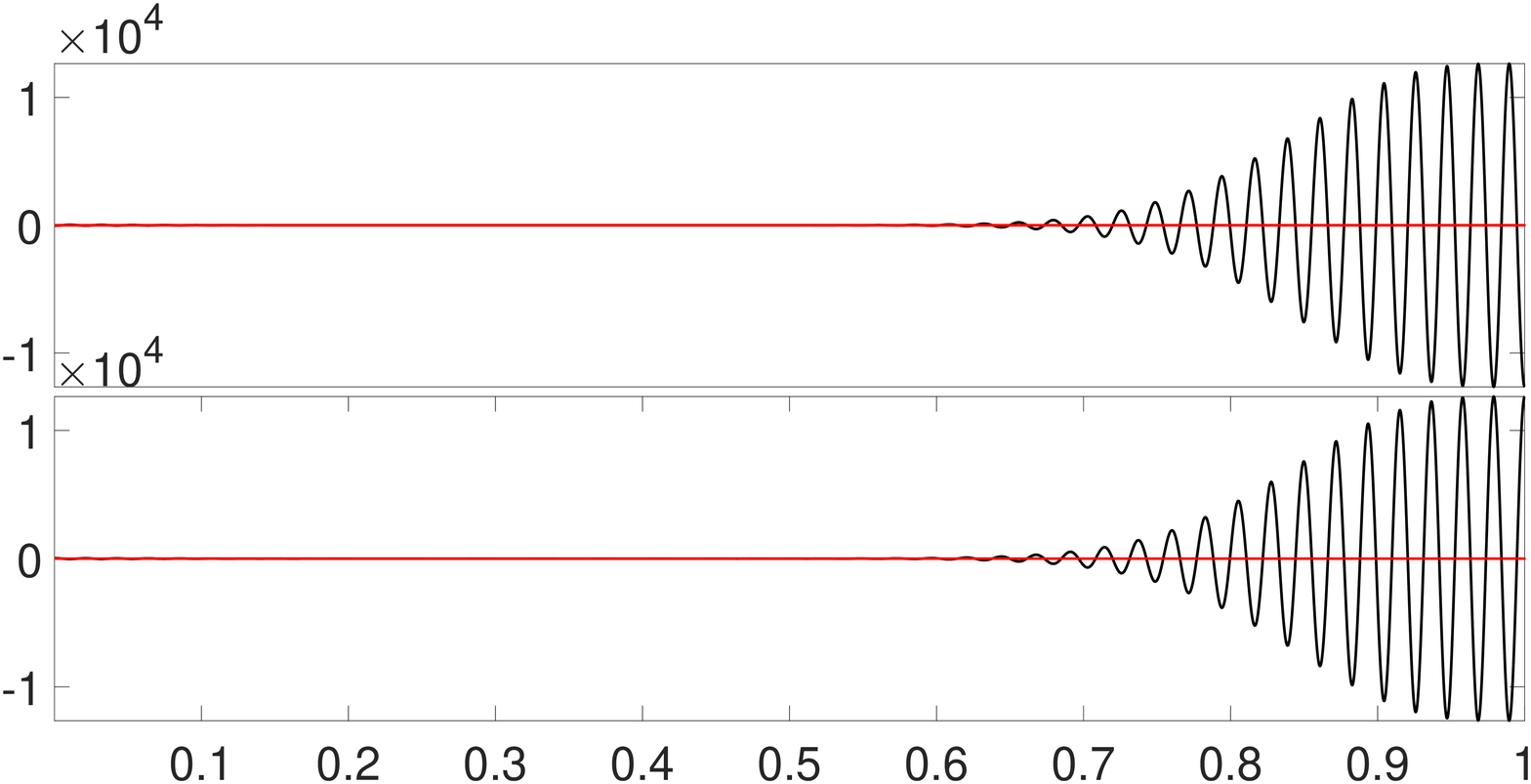}~\includegraphics[width=0.34\linewidth]{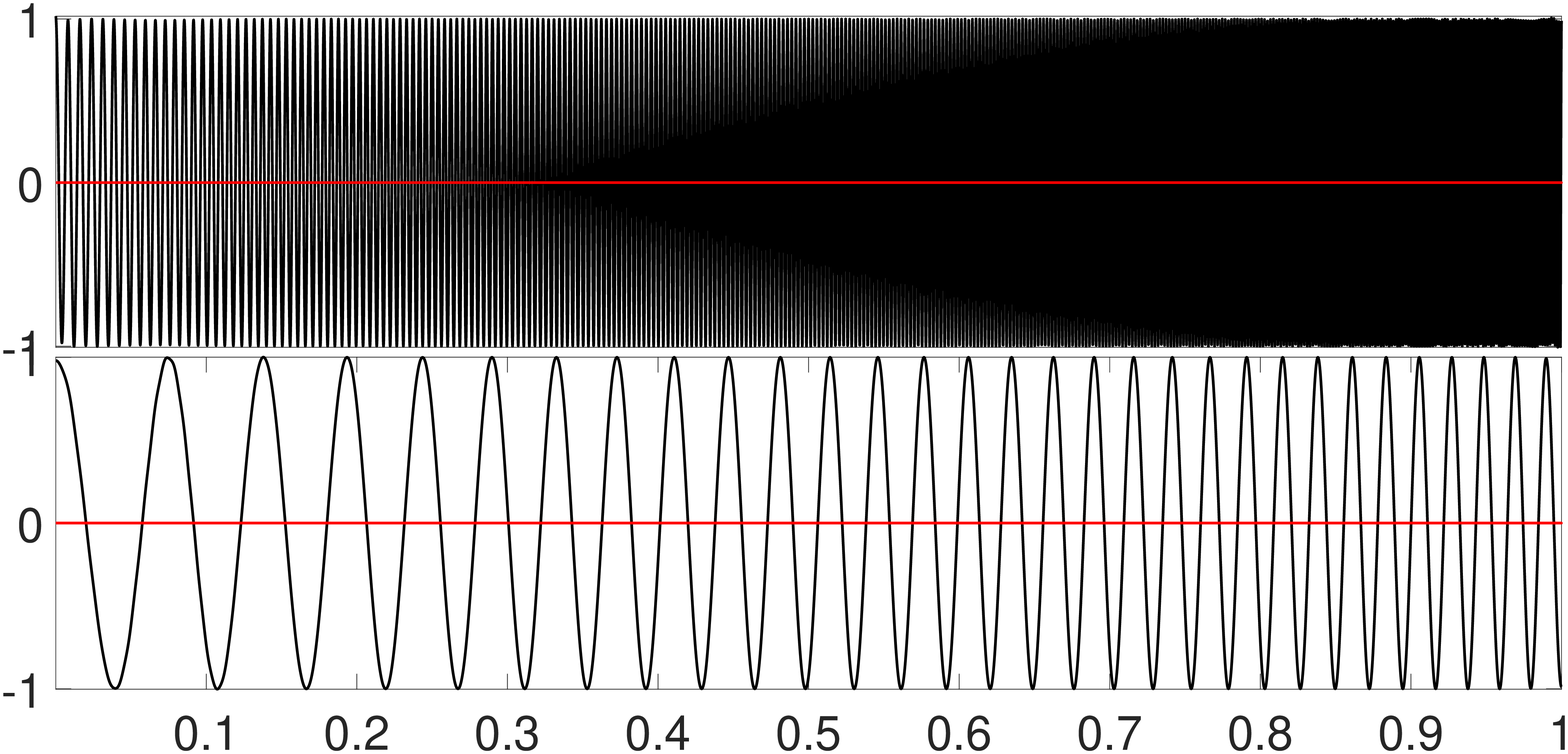}
\caption{First example of unwanted oscillations: FIF run with a bad choice of the filter. Left panel, signal. Central panel, FIF decomposition using a Fokker-Planck filter without convolving it with itself. Right panel, FIF decomposition using a filter obtained as convolution of a Fokker-Planck filter with itself.}
\label{fig:Unwanted_Case_1_fig1}
\end{figure}

\begin{figure}
\centering
\includegraphics[width=0.49\linewidth]{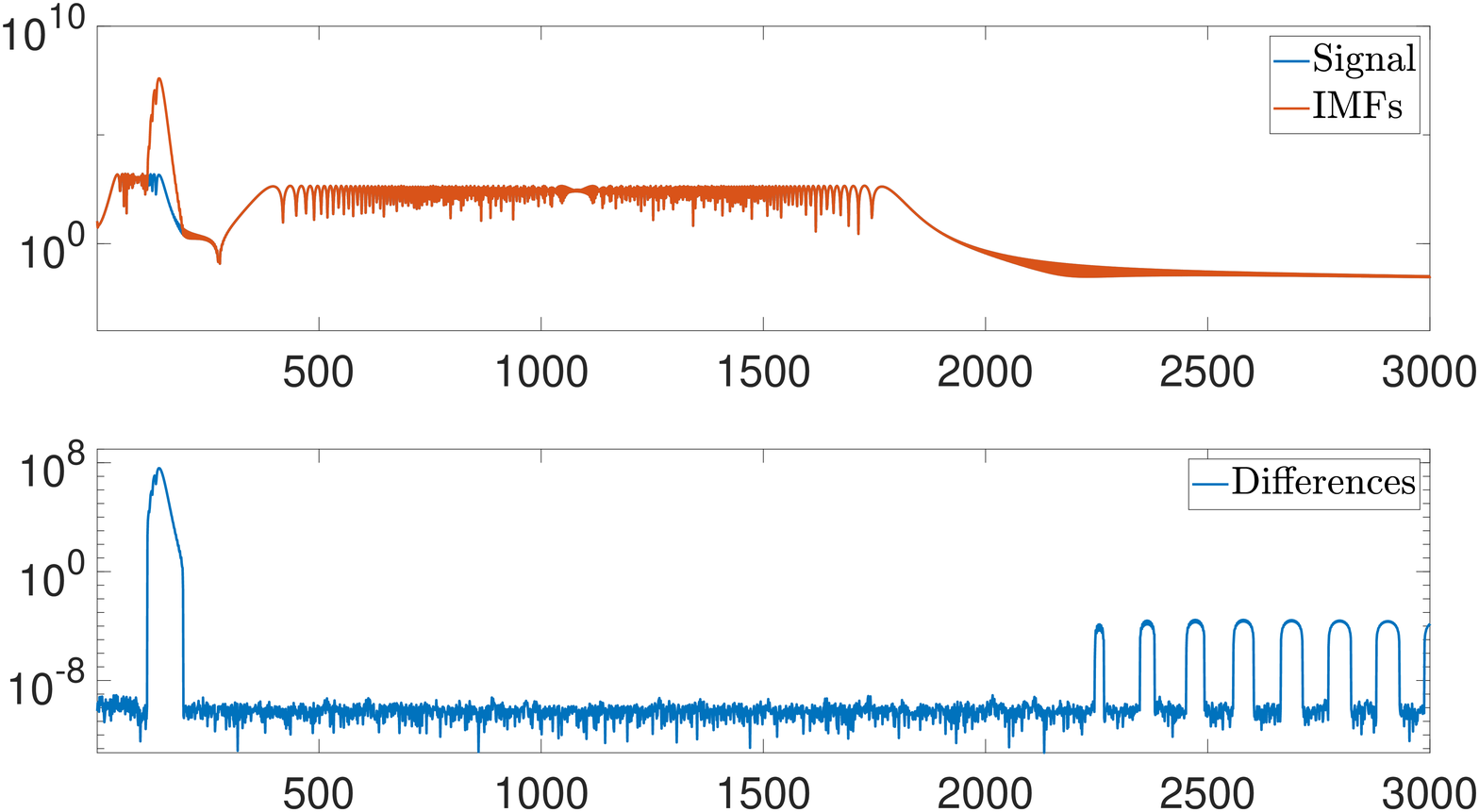}~\includegraphics[width=0.49\linewidth]{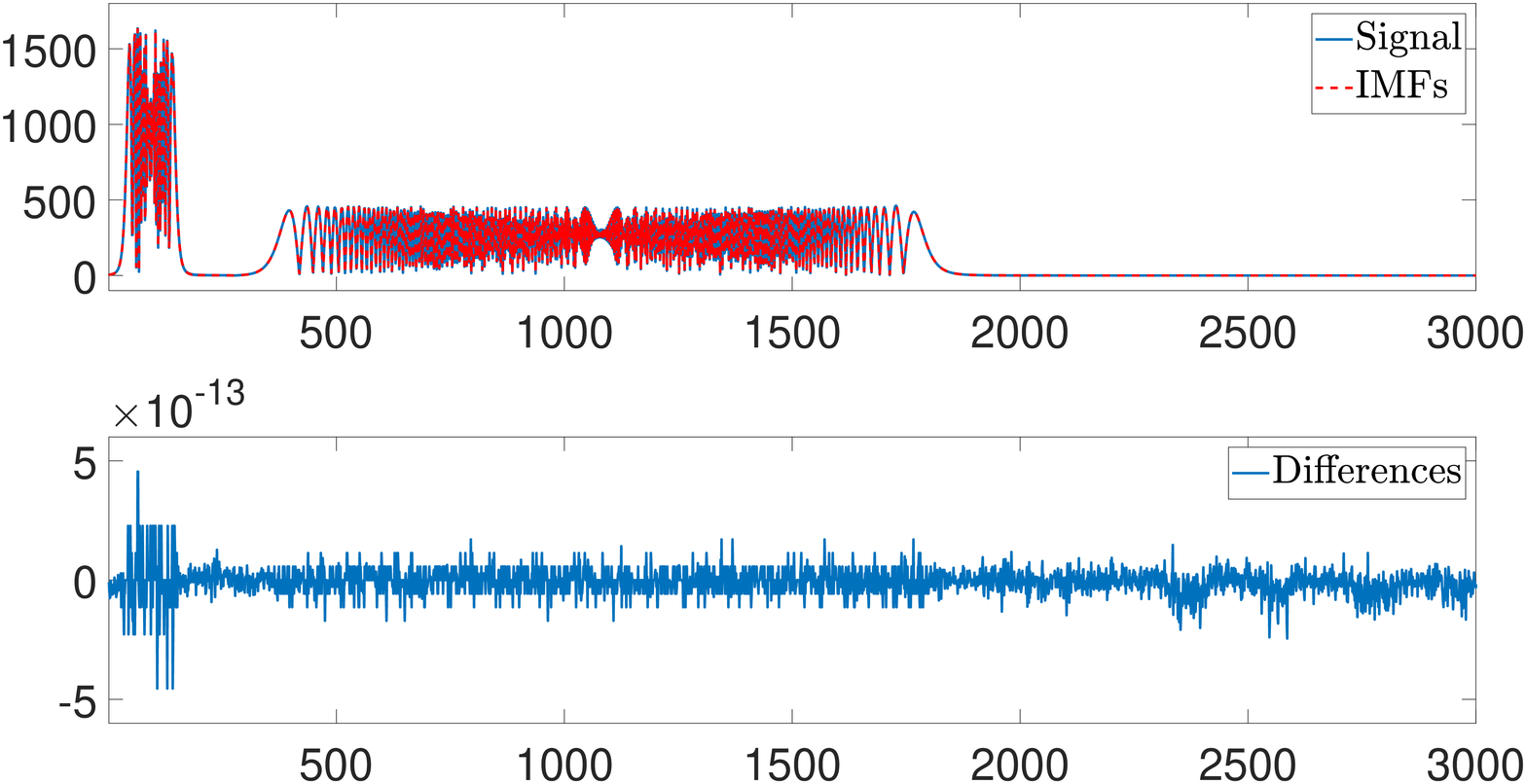}
\caption{First example of unwanted oscillations. First row: the absolute value of the signal FFT compared with the sum of the absolute values of the IMFs FFT. Left panel, FIF with a bad choice of the filter, right panel, FIF with a double convolution filter.
Second row the differences between the absolute value of the signal FFT and the sum of the absolute values of the IMFs FFT, left when the filter does not guarantee the convergence of FIF, right when the filter is chosen appropriately.}
\label{fig:Unwanted_Case_1_fig2}
\end{figure}

In Figure \ref{fig:Unwanted_Case_1_fig1}, left panel, we plot the signal $s(t)$ obtained as composition of the two non-stationary components
\begin{eqnarray}\label{eq:Example1}
\nonumber  s_1(t) &=& \cos\left(480\pi t^2+240\pi t\right)\\
           s_2(t) &=& \cos\left(\pi t^3+36\pi t^2+24\pi t+2\pi\right),
\end{eqnarray}
where $t\in [0,\,1]$.

In Figure \ref{fig:Unwanted_Case_1_fig1}, central panel, it is shown the FIF decomposition into two IMFs obtained using a Fokker-Planck filter \cite{cicone2016adaptive}, which has not been convolved with itself.
It is already evident from this plot that the algorithm is not converging to a meaningful solution and that unwanted oscillations become present. If we look at the absolute values of the Fourier coefficients of the original signal versus the summation of the absolute values of the Fourier coefficients of two IMFs produced, Figure \ref{fig:Unwanted_Case_1_fig2} left column, we have the confirmation that the $L_1$ Fourier Energy of the signal is clearly not conserved in this case.

If, instead, we consider the same Fokker-Planck filter used in the previous passage and, before applying it in the FIF algorithm, we convolve it with itself, then, by Theorem \ref{thm:Energy_conservation}, we are guaranteed a priori that the $L_1$ Fourier Energy of the signal will be conserved and, even more importantly, no unwanted oscillation will appear. This is confirmed by looking at the numerical results plotted in Figure \ref{fig:Unwanted_Case_1_fig1} right panel, where the newly produced IMFs are depicted and no unwanted oscillations are visible naked eyes. Furthermore, the absence of unwanted oscillations is confirmed also by looking at the absolute values of the Fourier coefficients of the original signal versus  the summation of the absolute values of the Fourier coefficients of these two newly obtained IMFs, Figure \ref{fig:Unwanted_Case_1_fig2} right column. The differences between these two curves, which is plotted in the bottom row, second column of Figure \ref{fig:Unwanted_Case_1_fig2}, is around machine precision.

\subsubsection{Example of unwanted oscillations: case of EMD decomposition}

The previous definition of unwanted oscillations has been tailored on IF--based technique decompositions, but it can be tested also on other signal decompositions. For instance, EMD algorithm is well known to be unstable in presence of noise \cite{huang2009EEMD}. In this example, we compare the FIF and EMD decomposition of a nonstationary signal $s(t)=s_1(t)+s_2(t)$, where $s_1$ and $s_2$ are defined in \eqref{eq:Example2}, perturbed by a component $s_3(t)$, which is stationary in frequency at 200 Hz, and whose amplitude $A(t)$ is a white Gaussian noise distribution whose four statistical momenta are $\mu=0$, $\sigma=0.18$, skewness 0.17 and kurtosis 6.
\begin{eqnarray}\label{eq:Example2}
\nonumber  s_1(t) &=& \cos\left(50\pi t^2+100\pi t\right)\\
           s_2(t) &=& \cos\left(-10\pi t^2+40\pi t\right)\\
           s_3(t) &=& A(t) \cos(400\pi t)
\end{eqnarray}
where $t\in [0,\,1]$.

\begin{figure}
\centering
\includegraphics[width=0.29\linewidth]{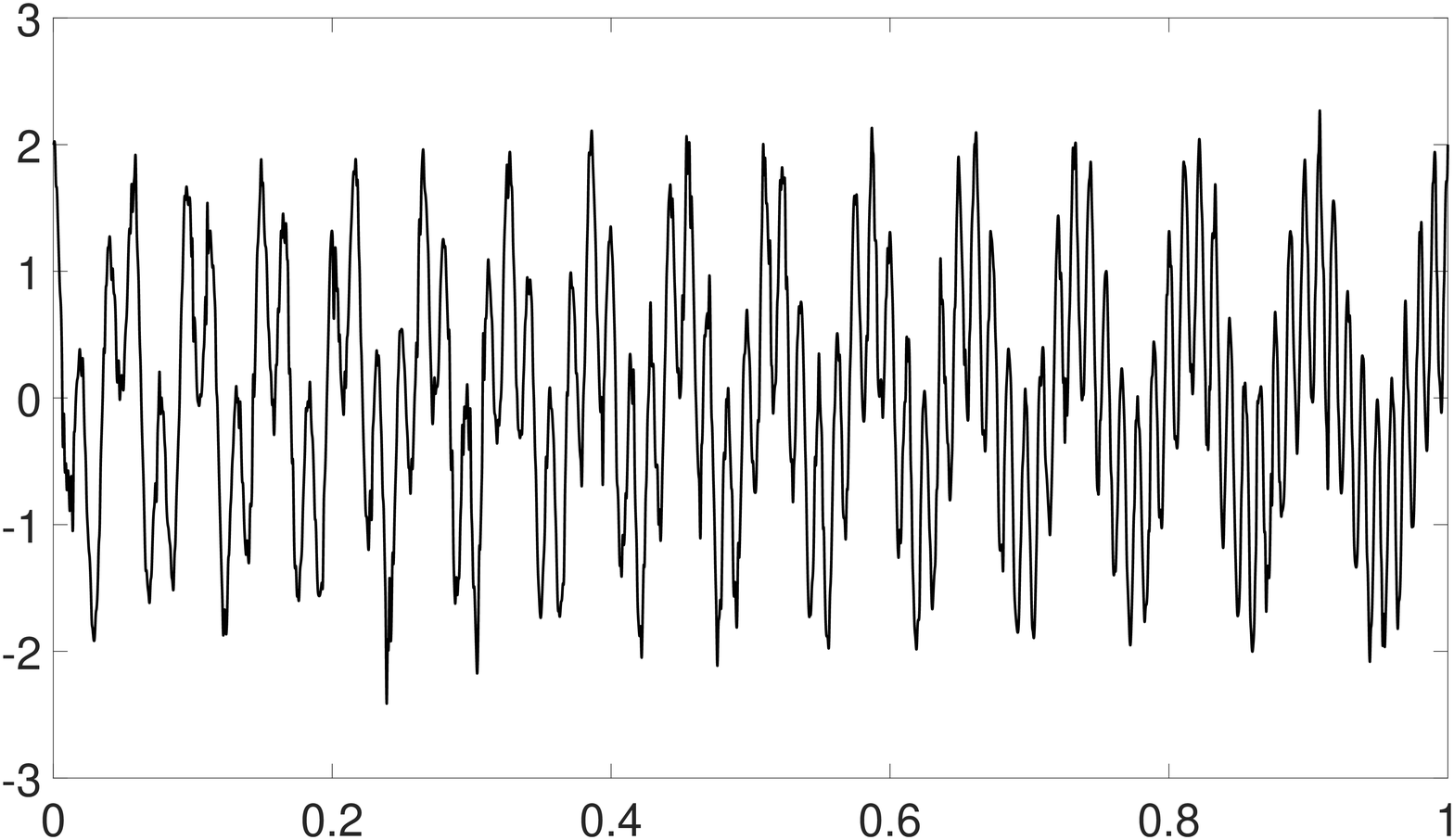}~\includegraphics[width=0.34\linewidth]{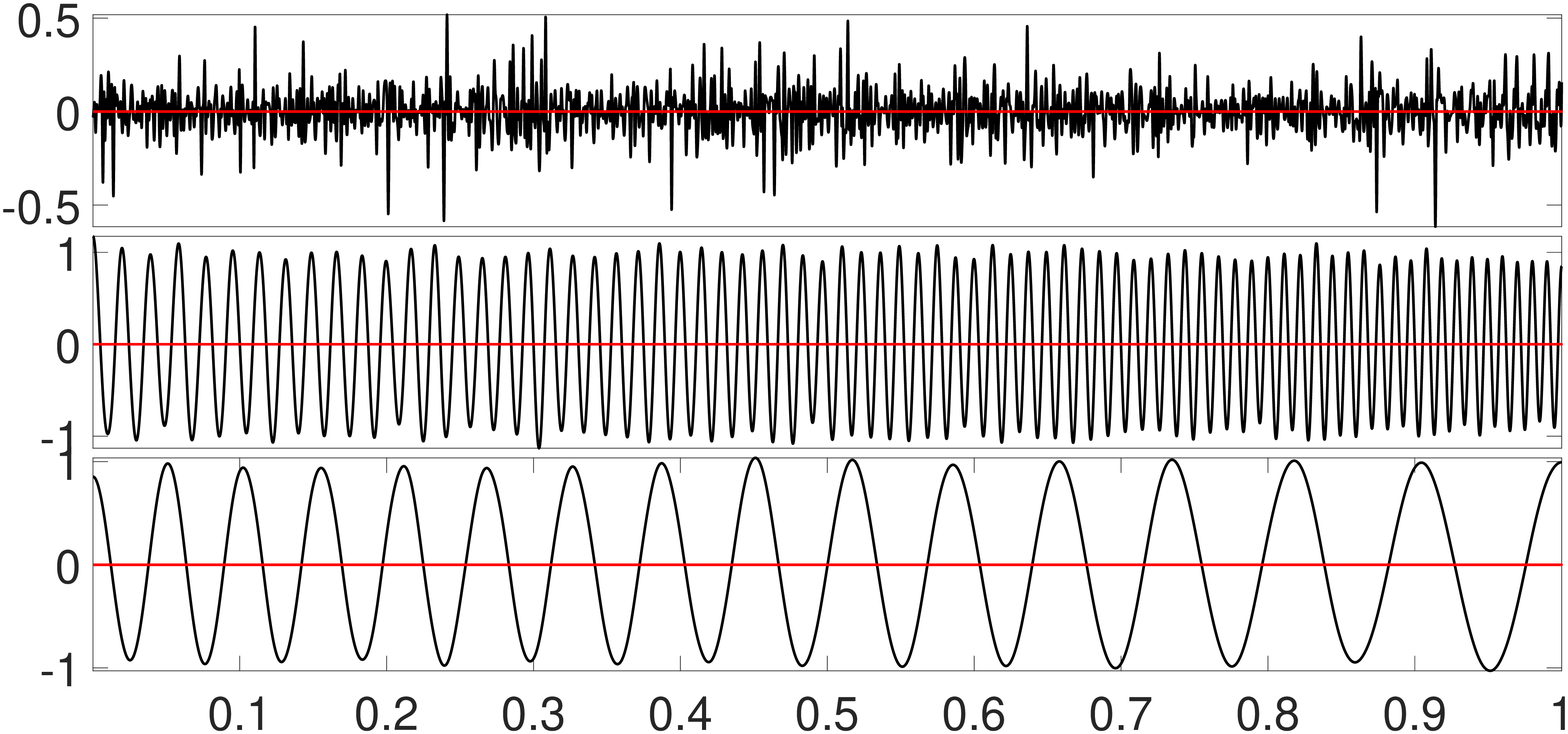}~\includegraphics[width=0.34\linewidth]{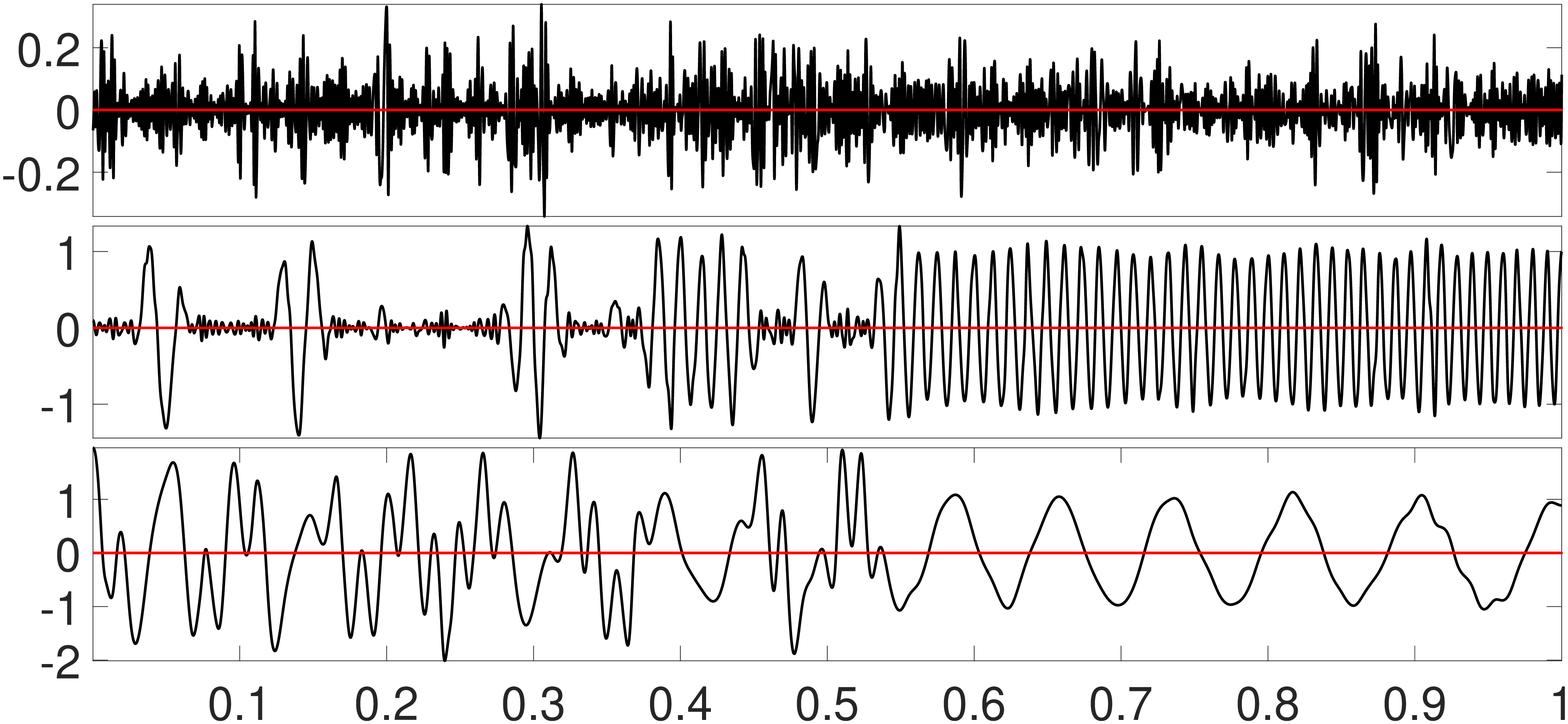}
\caption{Second example of unwanted oscillations. Left panel, the signal. Central panel, FIF decomposition with a double convolution filter. Right panel, EMD decomposition.}
\label{fig:Unwanted_Case_2_fig1}
\end{figure}

\begin{figure}
\centering
\includegraphics[width=0.5\linewidth]{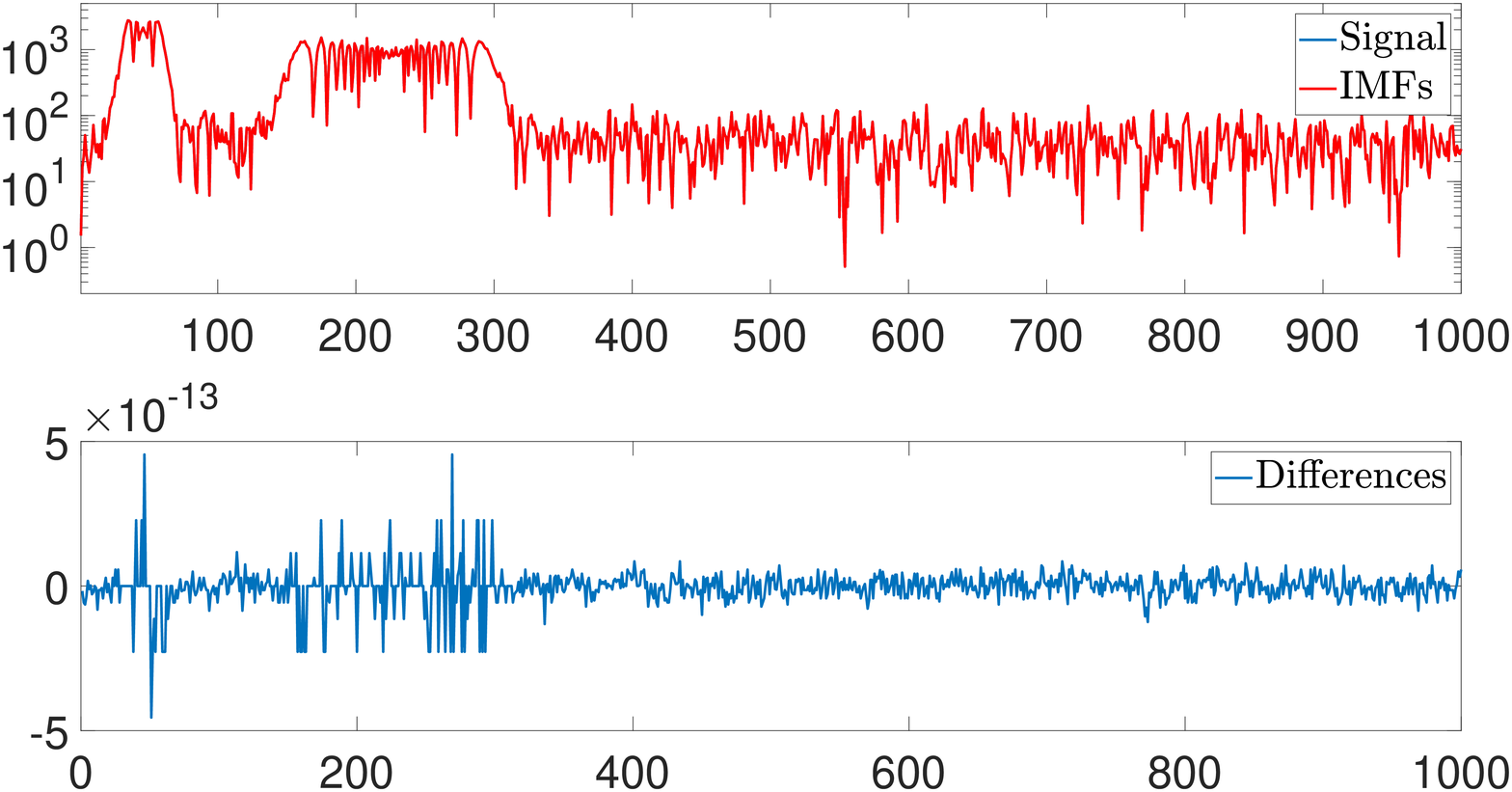}~\includegraphics[width=0.5\linewidth]{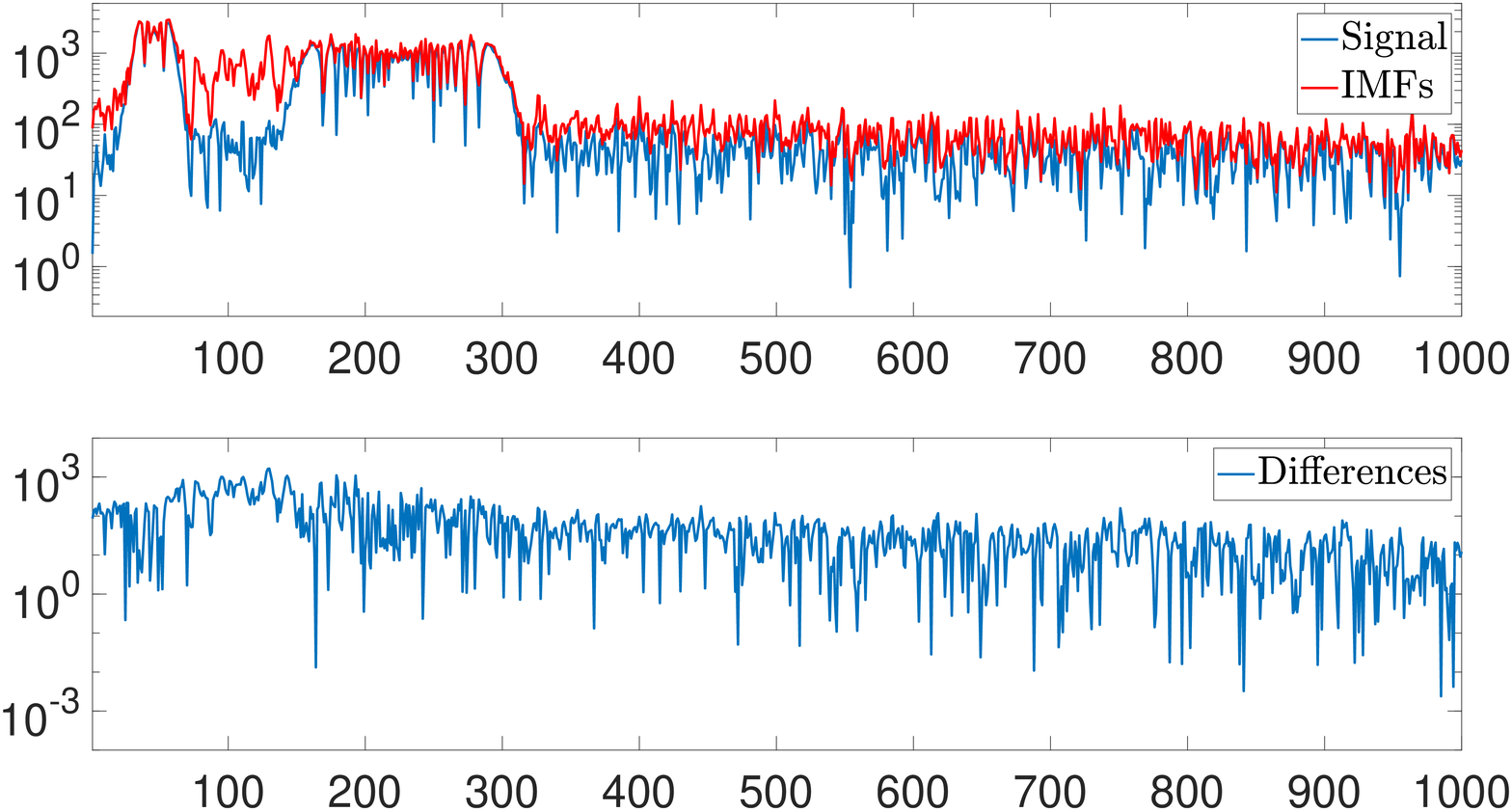}
\caption{Second example of unwanted oscillations. The absolute value of the signal FFT compared with the sum of the absolute values of the IMFs FFT when we apply FIF or EMD decomposition, respectively left and right columns.}
\label{fig:Unwanted_Case_2_fig2}
\end{figure}

From the rightmost panel of Figure \ref{fig:Unwanted_Case_2_fig1}, we can see that a mode mixing problem arises when we apply EMD to decompose the signal. This is due to the local instability of the decomposition that affects the classical EMD method \cite{huang2009EEMD}. Many generalizations of the EMD have been proposed in the last decade to fix this issue, as we mentioned in the introduction. It is important to recall here that the IF and FIF algorithms are not prone by construction to mode mixing. In Figure \ref{fig:Unwanted_Case_2_fig2} we plot the absolute values of the signal FFT as well as the sum of the absolute values of the IMFs FFT, for both the FIF and EMD decomposition, as well as their differences, bottom row. From these results, it is evident that FIF, as expected from the theory, conserves the $L_1$ Fourier Energy of the signal, whereas EMD decomposition is not conserving it. However, we remind that this result does not allow us to derive any further conclusion on the EMD, more than the mode mixing is reflecting, as expected, also in frequency domain.

\subsection{On the time-frequency representation of the IMFs}

In this section, we present the so called IMFogram technique for the time-frequency representation of signals, which was first introduced in \cite{Barbe2022time}. Here we will give a more in-depth analysis of IMFogram and comparisons to other time-frequency representations.

It is by design that each IMF has a small range of frequency with relatively large amplitude. With this in mind, we start by defining the \emph{local amplitude}, and \emph{local frequencies} of an IMF.

From now on we make the assumption that each component contained in the signal presents only interwave modulation, i.e. modulation from one period to the next one, in frequency and amplitude and not intrawave modulation, i.e. modulation inside each period. This is a strong assumption, however it allows to build a first kind of time frequency representation based on IMFs which, as we will prove, converges to the spectrogram representation for a certain class of signals. In fact classical methods based on Fourier and wavelet transforms have implicitly the same limitation. Developing a time-frequency representation based on IMFs that can capture the intrawave modulations contained in them remains an open problem which deserve to be studied in a future work.

\begin{definition}[Instantaneous amplitude and Instantaneous frequency function for an IMF]\label{def:IA_IF}
Fix a $\IMF=(\textrm{IMF}(t_j))_0^{n-1}$.
\begin{itemize}
\item[(1)]
Let $g$ be the linear interpolation of the local maximum of $|\textrm{IMF}|$. The {\it instantaneous amplitude  function} for this IMF, is defined as
$iA^{\IMF}:[t_0,t_{n-1}]\to \mathbb{R}$,  the linear interpolation of $(t_j, iA^{\IMF}(t_j))$, where $iA^{\IMF}(t_j)=\max\{g(t_j), |\IMF(t_j)|\}$.

\item[(2)] Let $f$ be the linear interpolation of $\IMF(t_j)$ over the interval $[t_0, t_{n-1}]$, and denote the zero crossings of $f$ by $\{z_1< z_2< \dots<z_p\}$. For $j=1, 2, \dots, p$, define $y_j=\frac{1}{z_{j+1}-z_j}$, $z_0 = t_0$, and $z_{p+1}=t_{n-1}$, $y_0=y_1$, $y_{p+1}=y_p$. The {\it instantaneous frequency function}, $iF^{\IMF}:[t_0,t_{n-1}]\to \mathbb{R}$ is the linear interpolation of $\{(z_j, 2y_j)\}_{j=0}^{p+1}$.
\end{itemize}
The instantaneous amplitude and instantaneous frequency of $\IMF$ at $t_j$ is $iA^{\IMF}(t_j)$ and $iF^{\IMF}(t_j)$ respectively. By slightly abuse of notation, if we identify the discrete sequence $\IMF$ with its linear interpolation function, the instantaneous amplitude and instantaneous frequency of $\IMF$ at $t\in [t_0, t_{n-1}]$ is simply $iA^{\IMF}(t)$ and $iF^{\IMF}(t)$ respectively.
\end{definition}

We observe that the choice of the envelope in the computation of the instantaneous amplitude and frequency is not stringent for the IMFogram method, since all is needed in this context is the computation of a local amplitude and frequency which is produced as an average of the instantaneous values. In the following, to construct the envelopes, we use linear interpolation for simplicity.

As one knows, the concept of instantaneous frequency and amplitude are not universally accepted. The purpose of the last definition is to fix the terminology and not to start a philosophical discussion. Our goal is to use the instantaneous frequency and amplitude to define a version of  local frequency and amplitude.

Fix the size of the time window to be the positive integer $J$.
Consider the $n\times n$ matrix (for simplicity, we assume that $n$ is a multiple of $J$) $T$, the averaging matrix, defined as
\[ T_{i,k} = \begin{cases}
1/J, \qquad 1+(p-1)J\le i, k\le pJ, \text{ for } p = 1, 2, \dots, n/J\\
0, \qquad otherwise. \end{cases}
\]
where $T_{i,k}$ is the $(i, k)$ - entry of $T$. The operator norm for $T$ is $\|T\|=1$, with respect to the normalized $2$-norm.

Similarly for the $2$-norm of a localized function. Let $s=(s_j)$ be a vector in $\mathbb{R}^n$ and $t_j$ is defined as in \eqref{eq:t_j} if $[a,b]\subset [0,L]$ in the time domain, with $a=t_i$ and $b=t_{i+k}$, then the local $2$-norm of $s$ over $[a,b]$ is
\begin{equation}
    \|s|_{[a,b]}\|_2=\frac{L}{n}\left( \sum_{0\le j\le k} |s(t_{i+j})|^2 \right)^{1/2}\label{eq:localnorm}
\end{equation}

Next, we define the local amplitude and local frequency of an IMF. In practice we observe that for any interval $[a,b]$ that is comparable to the filter length, say the length of $[a,b]$ is 2 to 5 times the length of the filter length, then $\|\IMF|_{[a,b]}\|_2$ and $\|iA^{\IMF}|_{[a,b]}\|_2$ are comparable. This is the base for the following definition of {\emph local amplitude}.

\begin{definition}[Local amplitude and Local frequency of an IMF]\label{def:LA_LF}
Given an $\IMF\in \mathbb{R}^n$ and let $iA^{\IMF}\in \mathbb{R}^n$ and $iF^{\IMF}\in \mathbb{R}^n$ be the  corresponding instantaneous amplitude and frequency, respectively. The average of each of them over a segment of length $J$ is the local amplitude $LA^{\IMF}$ and local frequency $LF^{\IMF}$, respectively. That is
\[ \displaystyle{LA^{\IMF}= T(iA^{\IMF})}, \quad
\displaystyle{LF^{\IMF} = T(iF^{\IMF})}
\]
\end{definition}

As we have observed above, $\|T\|= 1$, thus, we have
\[
\|LA^{\IMF}\|_2\le  \|iA^{\IMF}\|_2.
\]

Leveraging on the previous results we can introduce the TFR method named \emph{IMFogram}, whose pseudocode for the discrete setting is detailed in Algorithm \ref{algo:IMFogram}.
\begin{algorithm}
\caption{\textbf{IMFogram} $A$ = IMFogram(IMFs)}\label{algo:IMFogram}
\begin{algorithmic}
\STATE $M$ number of IMFs
\STATE $N$ signal length
\STATE $R$ number of overlapping time windows $I_j=[a_j,\ b_j]$
\FOR{ $k=1$ \TO $M$ }
     \STATE Compute the local amplitudes $LA^{(k)}_{I_j}$
     \STATE Compute the local frequencies $LF^{(k)}_{I_j}$
      \FOR{ $j=1$ \TO $R$ }
           \STATE \begin{equation}\label{eq:IMFogram}
                       A(LF^{(k)}_{I_j},j) = A(LF^{(k)}_{I_j},j)+LA^{(k)}_{I_j}\phantom{111111111111111111111111111111111111111}
                  \end{equation}
      \ENDFOR
\ENDFOR
\STATE return $A$
\end{algorithmic}
\end{algorithm}

The proposed TFR method requires as input the IMF decomposition of a given signal. In this work, based on the theoretical results proved so far in the literature \cite{cicone2020numerical}, we opt to use the IMFs produced by the FIF method. However, we point out, that any other alternative decomposition techniques proposed in the literature which produces IMFs can be used. As output the IMFogram produces a matrix $A$ which contains the local amplitudes of the various IMFs distributed by rows and columns depending on the corresponding local frequency and time window, respectively.

Regarding the choice of the time window, the IMFogram differs from the spectrogram, or similar TFR methods, since it does not require a clever choice of the time window length. The sliding time window is basically required in the IMFogram to reduce the number of columns of the output matrix $A$. The user, in fact, can choose as length of the sliding time window a value that goes from one sample point to the length of the signal itself.

Another important result concerning the FIF method is the following theorem proved in \cite{cicone2022oneortwo}
\begin{theorem}\label{thm:DIF_resolution}
Given the signal $s\in\R^p$ defined as $s(x_k, a, f) = \cos(2\pi x_k)+a \cos(2\pi f x_k+\phi), k\in[1,\ldots\ p]$, sampled at $\textrm{Fs}=\frac{1}{T}\gg 1$ samples/sec such that $p=n\frac{1}{T}$, where $f \in (\frac{1}{n}, 1)$, assuming $w$ is a double convolution filter, whose filter length measures $2L+1$, such that the smallest positive zero in the DFT of $w$ corresponds to frequency $1$.

Then FIF algorithm without stopping criterion can always resolve $s$ into the two components $\cos(2\pi x_k)$ and $a \cos(2\pi f x_k+\phi)$, as far as $f \leq 1-\frac{1}{n}$.
\end{theorem}

As an immediate corollary of this theorem we have that
 \begin{corollary}\label{cor:DIF_resolution}
Given a multicomponent signal $s = \sum_{j=1}^N a_j \cos(2\pi f_j x_k+\phi_j)$, $k\in[1,\ldots\ p]$, sampled at $\textrm{Fs}=\frac{1}{T}\gg 1$ samples/sec such that $p=n\frac{1}{T}$, where $f_j\ll \frac{1}{2T}$, $\forall\, j=1,\ \ldots,\ N$, $\frac{f_{j+1}}{f_j} \in (\frac{1}{n}, 1-\frac{1}{n}), \forall\, j=1,\ \ldots,\ N-1$, and $\frac{a_j}{a_h}=\textrm{O}(1), \forall j,h=1,\ \ldots,\ N$. Assuming that to extract each IMF component we are using a double convolution filter $w_j$, whose filter length measures $2L_j+1$, such that the smallest positive zero in the DFT of $w_j$ corresponds to the frequency $f_j$.

Then FIF algorithm without stopping criterion can always resolve $s$ into $N$ IMF components $a_j \cos(2\pi f_j x_k+\phi_j),\, j=1,\ \ldots,\ N$.
\end{corollary}
The proof follows directly from the previous theorem and the observation that FIF, as any EMD--like methods, produce all the IMFs by subsequent subtraction from the signal, starting from the highest frequency component to the lowest one.

Based on these results, we are now ready to prove the following
\begin{theorem}
    Given a discrete nonstationary signal $s\in\R^p$, sampled at $\textrm{Fs}=\frac{1}{T}\gg 1$ samples/sec such that $p=n\frac{1}{T}$, and which we assume to be stationary on $K$ non-overlapping time windows $I_i$, $i=1,\, \ldots,\, K$, of length $L=\frac{p}{K}$, assuming that in each time interval $I_i$ the signal is given by $s(I_i) = \sum_{j=1}^{N_i} a_j^{(i)} \cos(2\pi f_j^{(i)} x_k+\phi_j^{(i)})$, $x_k\in I_i$, $\frac{f_{j+1}^{(i)}}{f_j^{(i)}}\in\left(\frac{1}{n},1-\frac{1}{n}\right)$, for every $j=1,\ \ldots,\ N_i-1$, and $\frac{a_j^{(i)}}{a_h^{(i)}}=\textrm{O}(1),\, \forall\, j,h=1,\ \ldots,\ N_i$, if we let the FIF stopping criterion $\delta$ to go to zero, then the Hadamard power two of the IMFogram matrix $A$ converges to the spectrogram matrix obtained as the entry-wise squared absolute value of the DFT of the matrix $S\in\R^{L\,\times\, K}$, which contains in each column the non-overlapping time windows $I_i$, $i=1,\, \ldots,\, K$, of length $L$ of the signal.
\end{theorem}
\begin{proof}
We start considering a single time windows $I_i$. From Corollary \ref{cor:DIF_resolution} follows that FIF, when $\delta\rightarrow 0$, ends up separating the signal $s(I_i) = \sum_{j=1}^{N_i} a_j^{(i)} \cos(2\pi f_j^{(i)} x_k+\phi_j^{(i)})$ into the $N_i$ IMFs $a_j^{(i)} \cos(2\pi f_j^{(i)} x_k+\phi_j^{(i)})$, $j=1,\, \ldots,\, N_i$. Furthermore, based on Definition \ref{def:LA_LF}, the local frequency will be $LF_{I_i}^{(j)}=f_j^{(i)}$ and the local amplitude will be $LA_{I_i}^{(j)}=a_j^{(i)}$. Therefore, from Algorithm \ref{algo:IMFogram} it follows that the matrix $A$ will contain in the $i$-th column and at rows corresponding to frequencies $f_j^{(i)}$ the values $a_j^{(i)}$, $j=1,\, \ldots,\, N_i$. These amplitudes match the Fourier coefficients computed via Discrete Fourier Transform in the interval $I_i$ at frequencies $f_j^{(i)}$'s. The same result applies for all intervals $I_i$, $i=1,\, \ldots,\, K$. Hence, the entry-wise power two of the matrix $A$, i.e. $A$ matrix Hadamard power two, converges the spectrogram of the signals $s$ computed on $K$ non-overlapping time windows of length $L$ as the FIF stopping criterion $\delta$ goes to 0.
\end{proof}

We point out that, if we consider a time window in the IMFogram which matches the length of the signal, the proposed TFR method boils down to a periodogram of the signal.

It is important to highlight here that the TFR produced by IMFogram and spectrogram/periodogram are different in nature when the signal under study is not stationary inside time windows. In this more general scenario, their difference is due to the fact that the IMFogram uses information coming from the time domain, local periods and local amplitudes, to produce the TFR. Whereas, the spectrogram and periodogram, depending if we window or not the signal, use information coming from the frequency domain to produce the TFR. This explains why IMFogram can be more local in its TFR. The IMFogram does not need to extract frequency information from the Fourier transform of signal  windows. Furthermore, this windowing is the main cause of artifacts in the spectrogram. In fact, if on the one hand, the Heisenberg uncertainty principle cannot be eliminated, on the other hand the IMFogram allows to remove the uncertainties coming from the necessity, in the spectrogram and equivalent methods, of windowing the signal on sufficiently long time windows to reduce the local boundary effects induced by the Fourier Transform.

Another important observation regards how, in acoustics, a vibrato (frequency modulation with stationary amplitude) of a musical instrument can be perfectly reproduced by an expert musician as a amplitude modulation of a stationary frequency signal, what it is called a amplitude modulated monochromatic signal. With this knowledge, we observe that, while the TFR methods based on Fourier Transforms, like the well known periodogram and spectrogram, make use of stationary amplitude components to represent a nonstationary signal, the IMFogram, based on FIF--like decompositions, tends to represent a quasi--stationary frequency signal with modulated amplitude.
From this prospective, the IMFogram is a new and complementary way of looking to a signal with respect to standard TFR methods.

\section{Numerical Examples}\label{sec:NumericalExamples}

In this section we present few synthetic and real life examples of application of IMFogram algorithm as well as its comparisons with other classical methods. In particular, we compare IMFogram with spectrogram (STFT), continuous wavelet transform (CWT), synchrosqueezing transform (SST), and Hilbert Huang trasnform (HHT).

The following tests have been conducted using MATLAB$^{\tiny{\textregistered}}$ R2021a installed on a 64--bit Windows 10 Pro computer equipped with an Intel$^{\tiny{\textcopyright}}$ Core$^{\tiny{\textregistered}}$ i7-8550U at 1.80GHz CPU and 16GB RAM. All tested examples and algorithms are freely available online\footnote{\url{www.cicone.com}}.

\subsection{Synthetic example -- Undamped Duffing equation}

As a first example, we consider the synthetic signal generated as solution of the undamped Duffing equation
\begin{equation}\label{eq:duffing}
    \ddot{x}+\alpha x + \beta x^3 = \gamma \cos(\omega t),
\end{equation}
where $\alpha$ represents the linear stiffness, $\beta$ the nonlinearity in the restoring force, $\gamma$ the amplitude, $\omega$ the angular frequency of the forcing term. We can also rewrite \eqref{eq:duffing} to read
\begin{equation}\label{eq:duffing2}
    \ddot{x}+\left(\alpha + \beta x^2\right)x = \gamma \cos(\omega t),
\end{equation}
where $(\alpha + \beta x^2)$ is the nonlinear stiffness coefficient $k(x)$.

In Figure \ref{fig:Duffing} we plot the $\dot{x}$ solution of \eqref{eq:duffing2} when $ \alpha=-1$, $\beta = 1$, $\gamma = 0.1$, and $\omega= 1$.
The signal under investigation apparently contains two superimposed oscillatory components. In Figure \ref{fig:Duffing_TFR} we report the STFT, CWT and SST representation of $\dot{x}$. From all these TFR plots, it appears that three nonstationary frequency components are contained in the signal.
However, if we decompose, instead, the signal into IMFs, we discover that there are actually only two components contained in this signal, Figure \ref{fig:Duffing_TFR} second column. If we apply the HHT and the IMFogram, we obtain the TFRs plotted in the second column of Figure \ref{fig:Duffing_TFR}. Using these two plots it is possible to solve the riddle. The high frequency component is modulated in frequency at a high pace. In fact, the frequency of that component varies from period to period. It is evident from this example that it is not possible for methods based on wavelet or sinusoidal bases to capture properly such behavior. By comparing the HHT, right column second row in Figure \ref{fig:Duffing_TFR}, with the IMFogram, right column bottom row in Figure \ref{fig:Duffing_TFR}, we can see that both methods are able to produce a good TFR of the signal. However, it is important to mention that the HHT requires in this particular example an ad hoc tuning of the IMFs amplitudes to produce the result plotted in Figure \ref{fig:Duffing_TFR}. Otherwise, the first highest frequency component instantaneous frequency curve would not be visible in the plot. Whereas, the IMFogram did not require any special tuning to produce the result shown in Figure \ref{fig:Duffing_TFR}.

\begin{figure}
\centering
\includegraphics[width=0.8\linewidth]{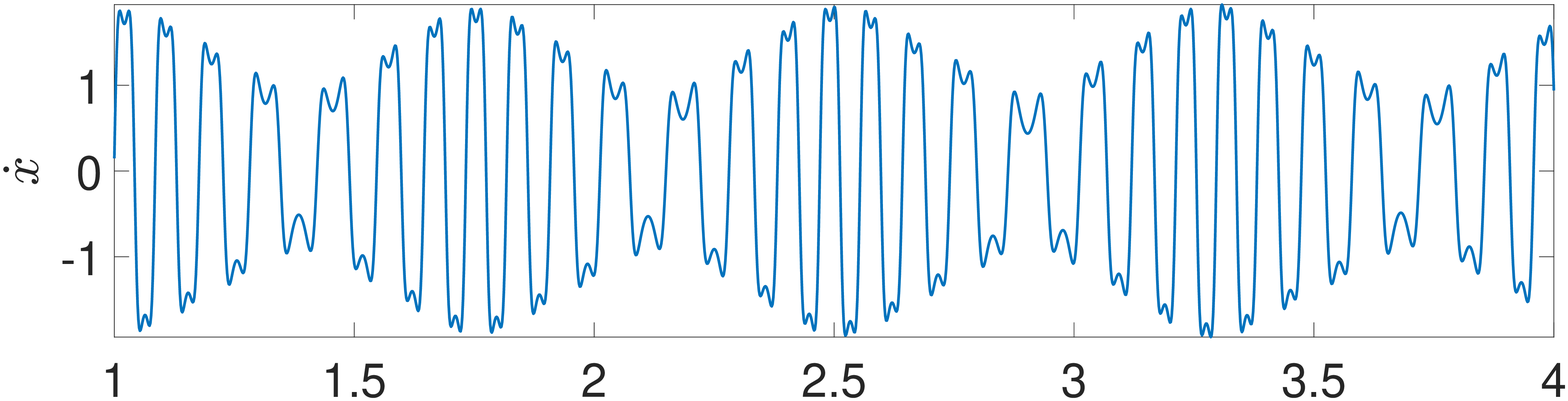}
\caption{Duffing equation example: the $\dot{x}$ solution of \eqref{eq:duffing2}.}
\label{fig:Duffing}
\end{figure}

\begin{figure}
\centering
\includegraphics[width=0.4\linewidth]{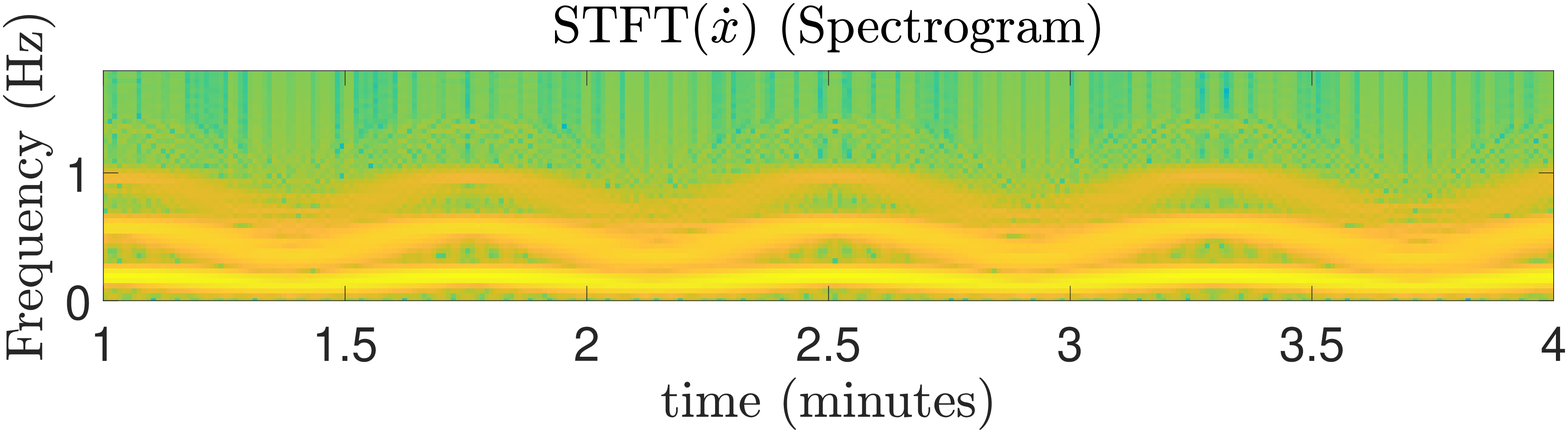}~$\phantom{11111111}$~\includegraphics[width=0.4\linewidth]{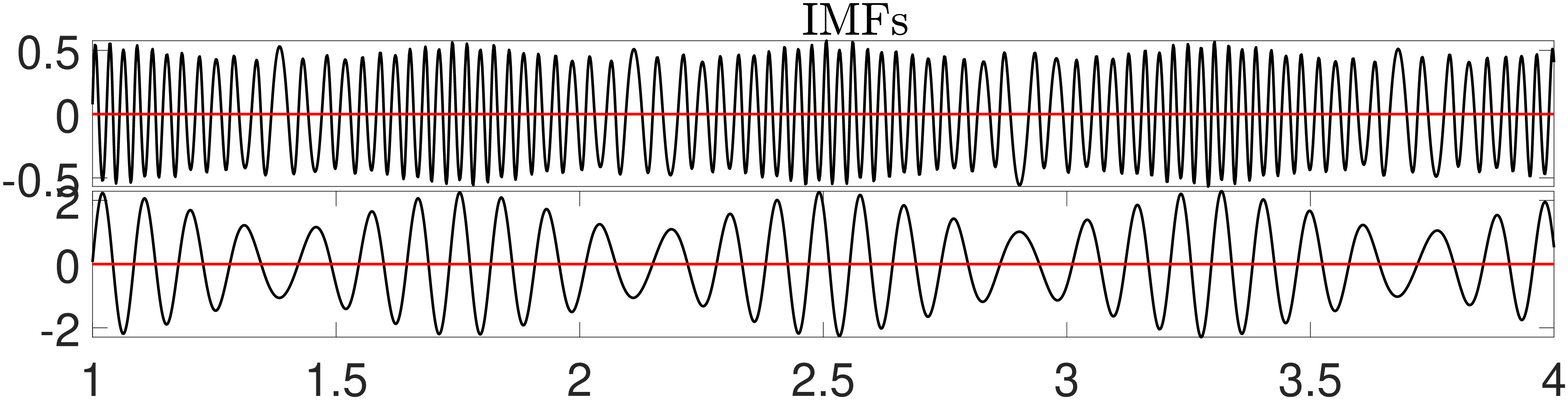}\\
\includegraphics[width=0.48\linewidth]{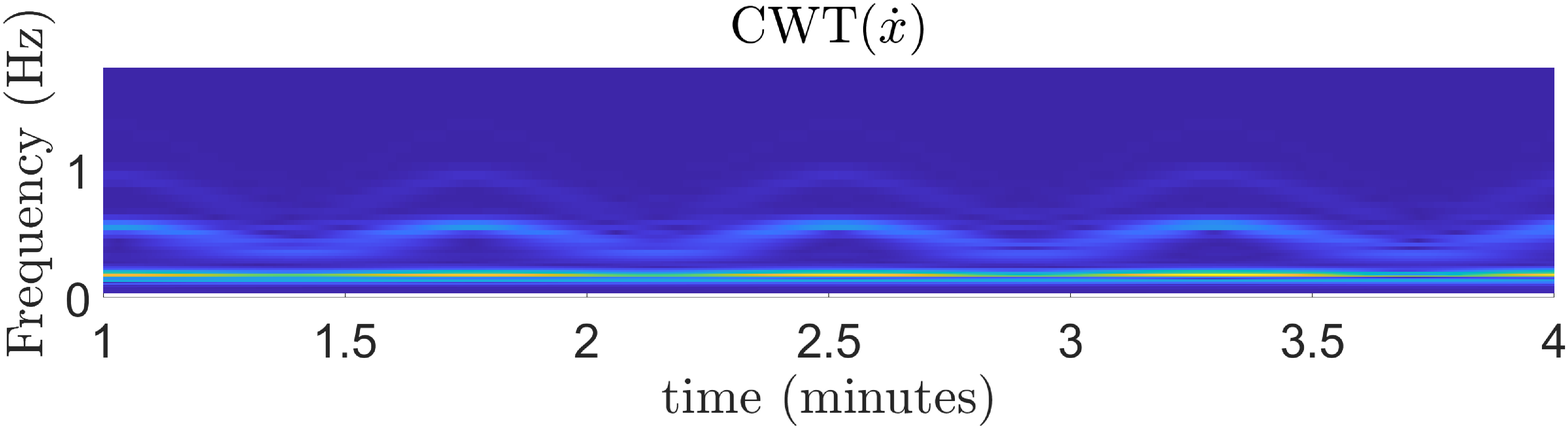}~\includegraphics[width=0.48\linewidth]{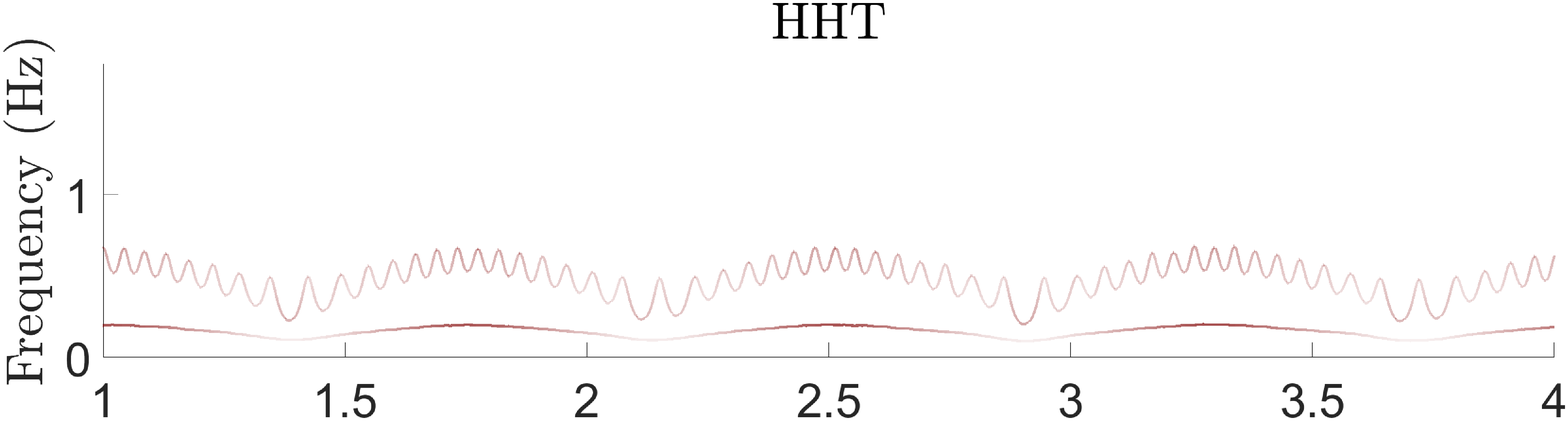}\\
\includegraphics[width=0.48\linewidth]{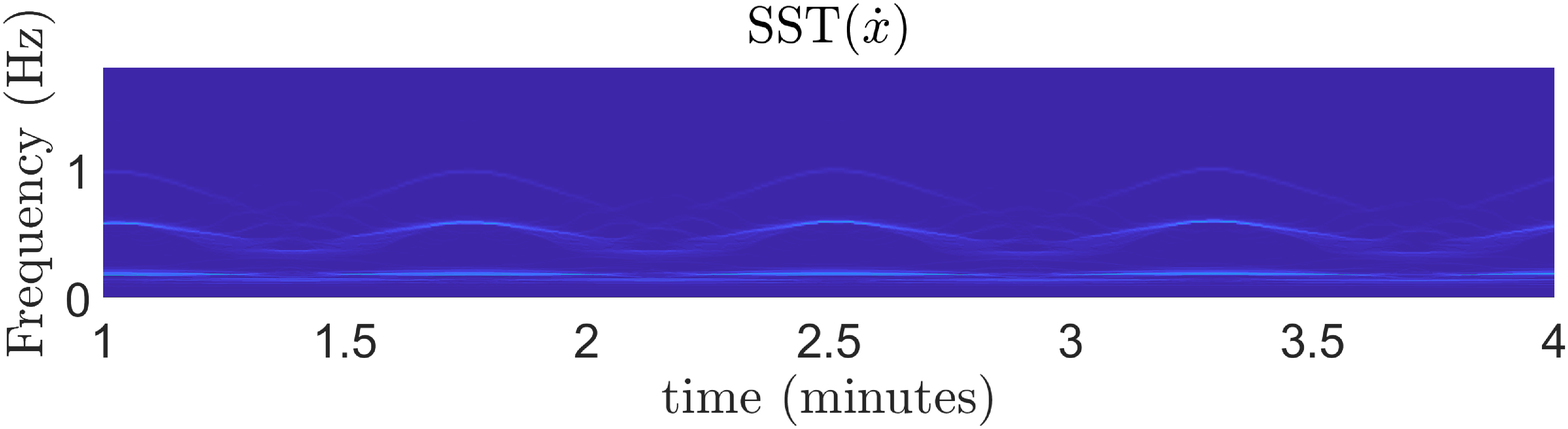}~\includegraphics[width=0.48\linewidth]{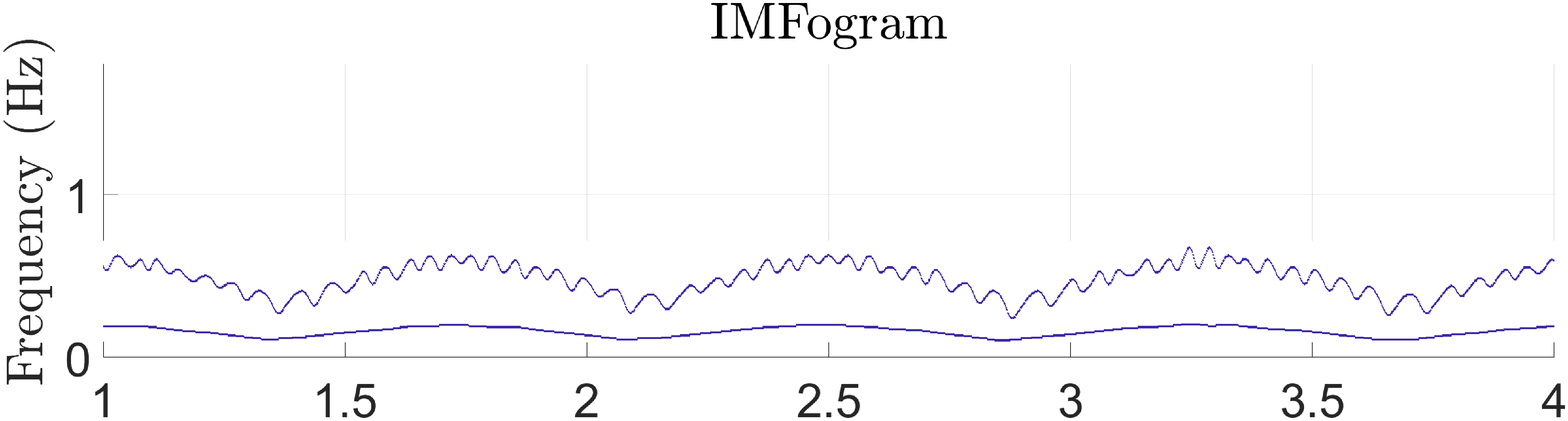}
\caption{Duffing equation example: In the first column, we plot the STFT, CWT, and SST representation of the $\dot{x}$ solution of \eqref{eq:duffing2}, respectively first, second and third row of the first column. Second column, we show the the Fast Iterative Filtering decomposition into IMFs of $\dot{x}$, the HHT and the IMFogram of these IMFs, respectively first, second and third row of the second column.}
\label{fig:Duffing_TFR}
\end{figure}

Furthermore, we underline that it is not sufficient to decompose the signal into IMFs to obtain a higher accuracy in its TFR. In Figure \ref{fig:DuffingIMFs} we show that, even if we first split the signal into two IMFs and then we apply the SST to them, still the TFR obtained contains the same artificial harmonics seen when we studied the whole signal, Figure \ref{fig:Duffing_TFR}.

\begin{figure}
\centering
\includegraphics[width=0.49\linewidth]{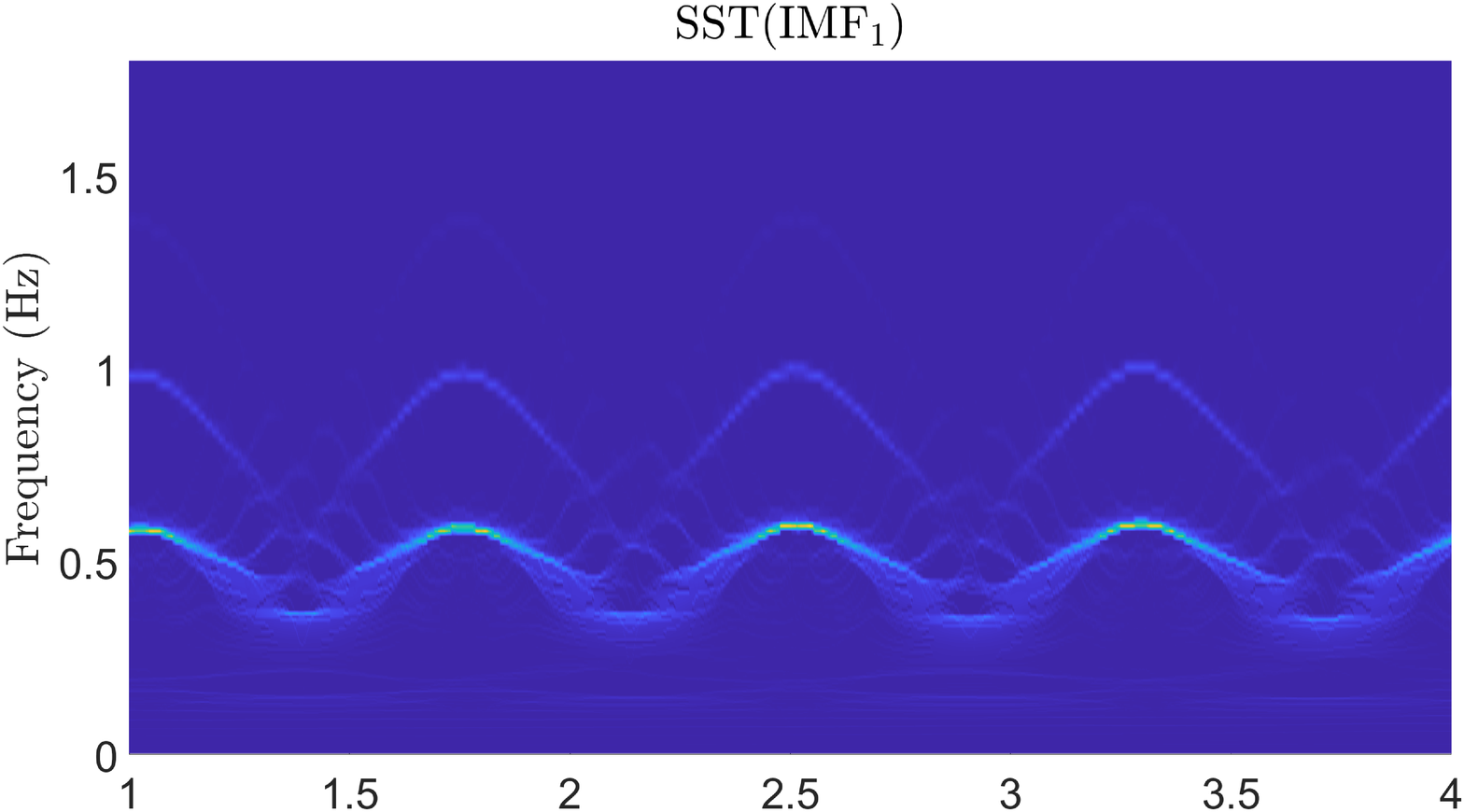}~\includegraphics[width=0.49\linewidth]{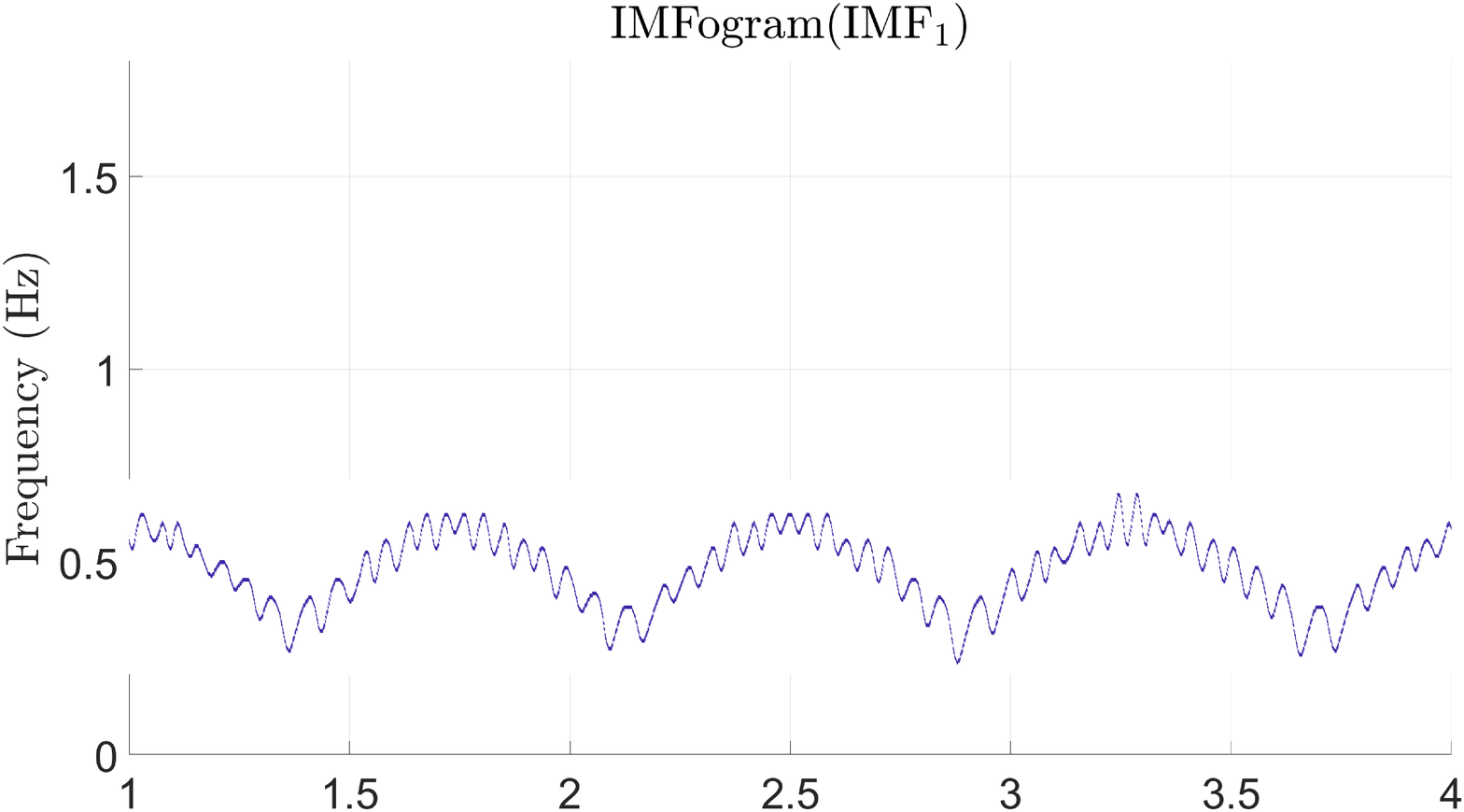}\\
\includegraphics[width=0.49\linewidth]{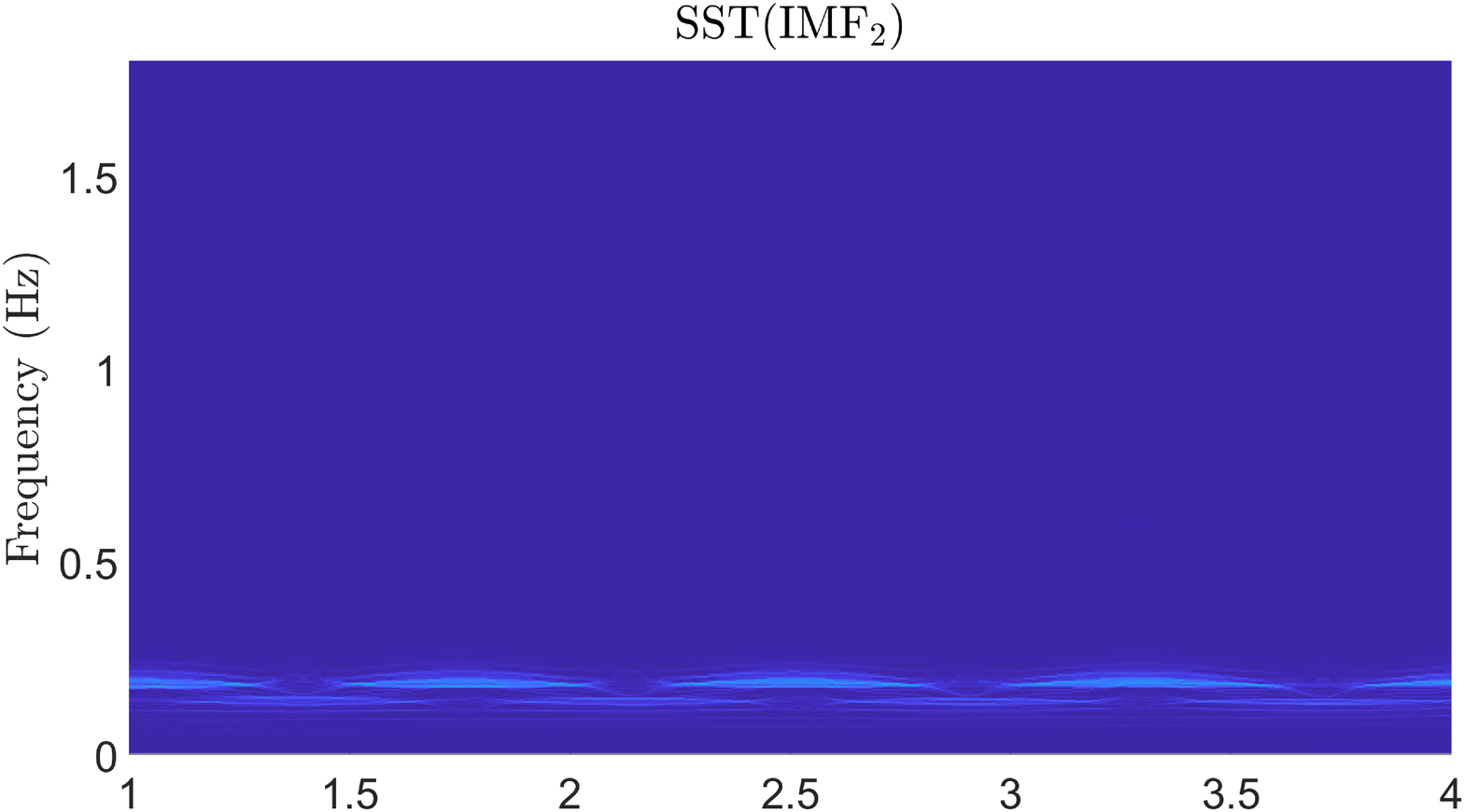}~\includegraphics[width=0.49\linewidth]{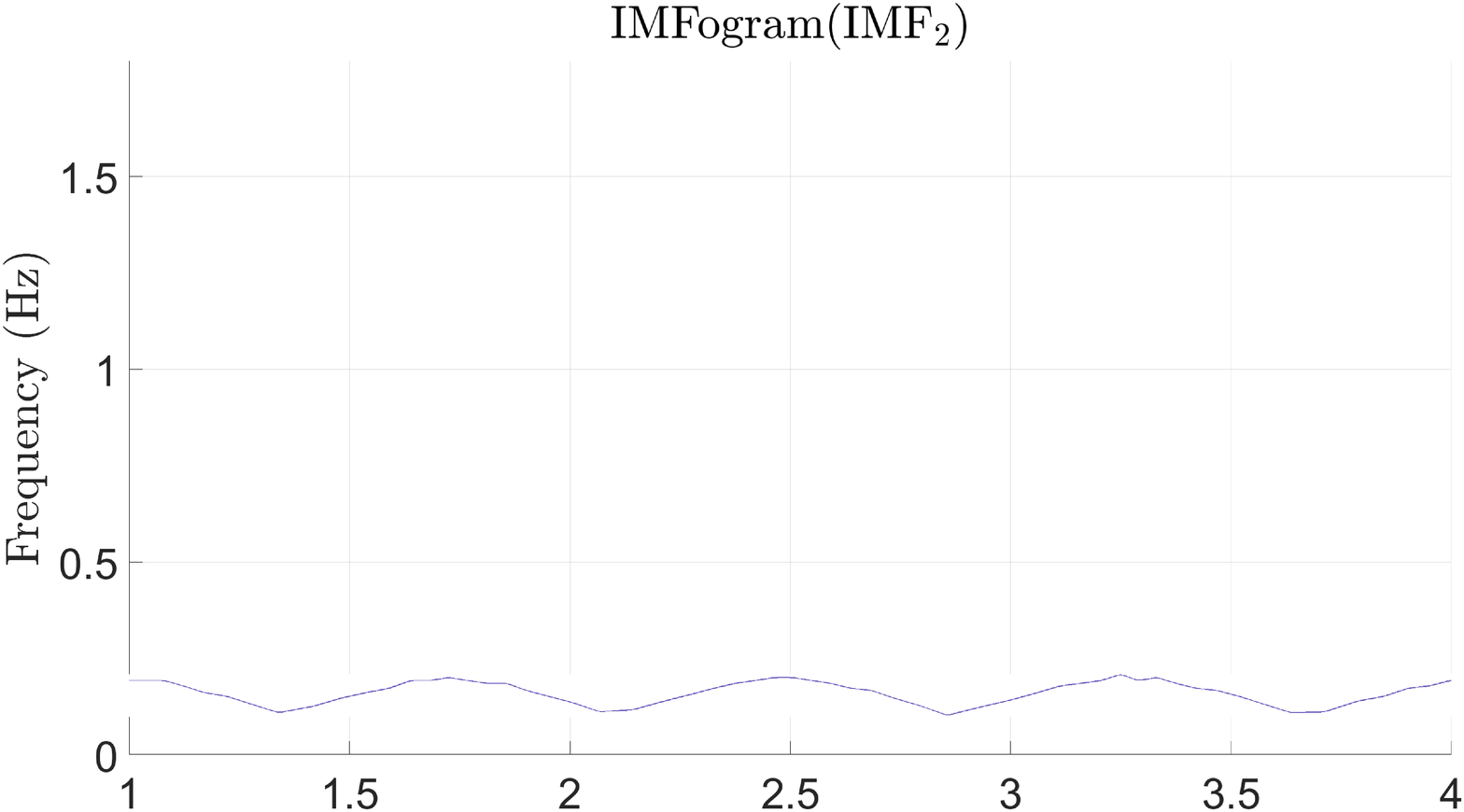}
\caption{Duffing equation example: Time Frequency representations of each IMF using SST (left column) and IMFogram (right column). First row $\IMF_1$, second row $\IMF_2$. }
\label{fig:DuffingIMFs}
\end{figure}

\subsection{Real life example -- violin vibrato}

In this real life example, we consider the recording of a violin when a G open string is played with vibrato\footnote{\url{https://en.wikipedia.org/wiki/File:Violin_vibrato_on_open_string_notes_and_on_fingered_notes.ogg}}. The signal is plotted in Figure \ref{fig:vibrato}, together with its STFT, CWT and SST representations. Furthermore, we report few IMFs from the signal decomposition produced using FIF algorithm. In the bottom row of this figure we report both HHT and IMFogram TFRs. The HHT is clearly having problems to properly capture the instantaneous frequencies contained in this signal. The IMFogram, instead, proves able to provide a crisp and clear representation of the time-frequency behavior of this signal which is, at a first view, comparable with the quality of the signal SST representation. However, at a more accurate analysis, we can observe how IMFogram contains many fine details that are completely missing in the SST representation. For instance, the oscillatory component around 800 Hz appears almost stationary in the SST representation, whereas in the IMFogram plot it has an instantaneous frequency which is clearly oscillating.

\begin{figure}
\centering
\includegraphics[width=0.40\linewidth]{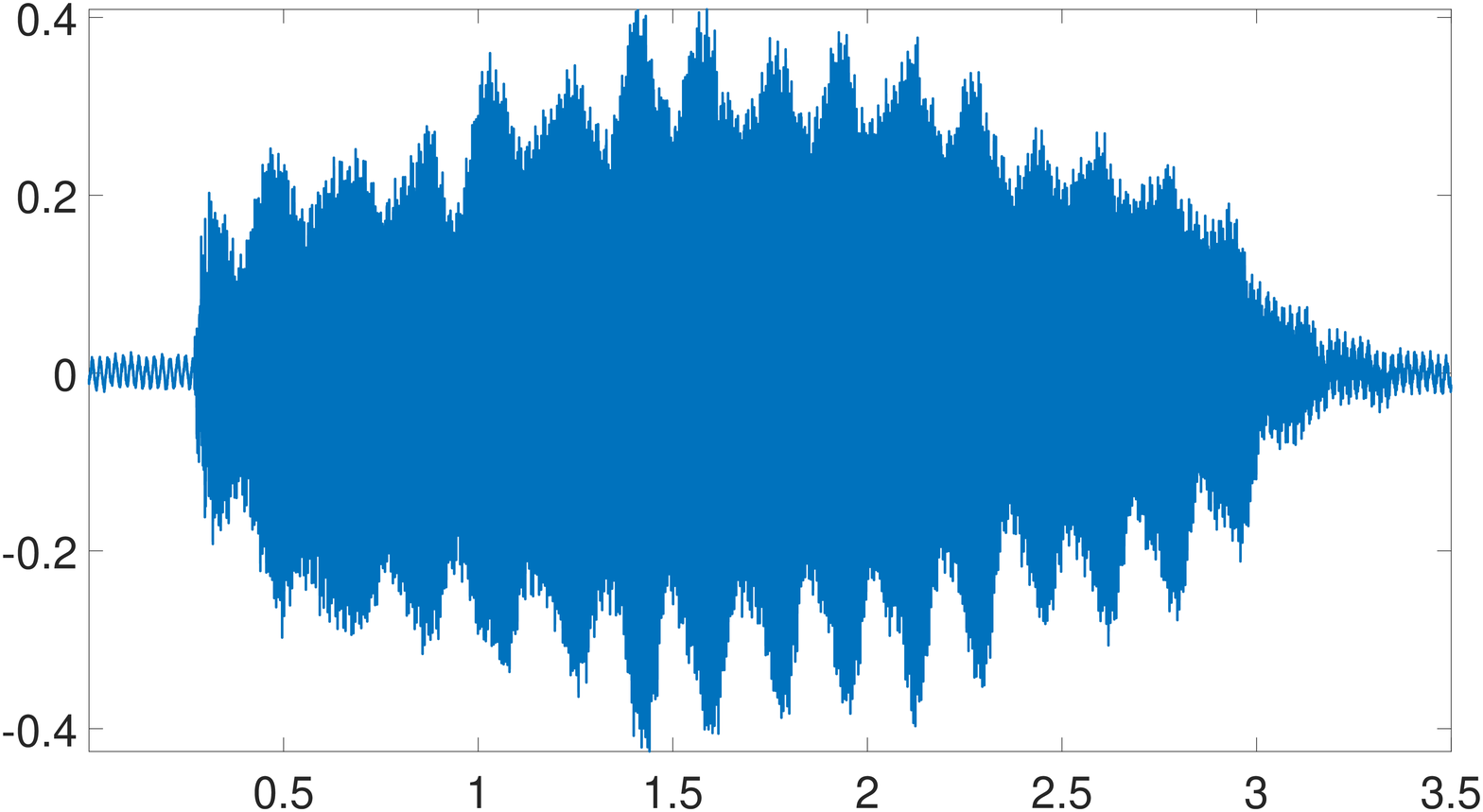}~$\phantom{1111111}$~\includegraphics[width=0.40\linewidth]{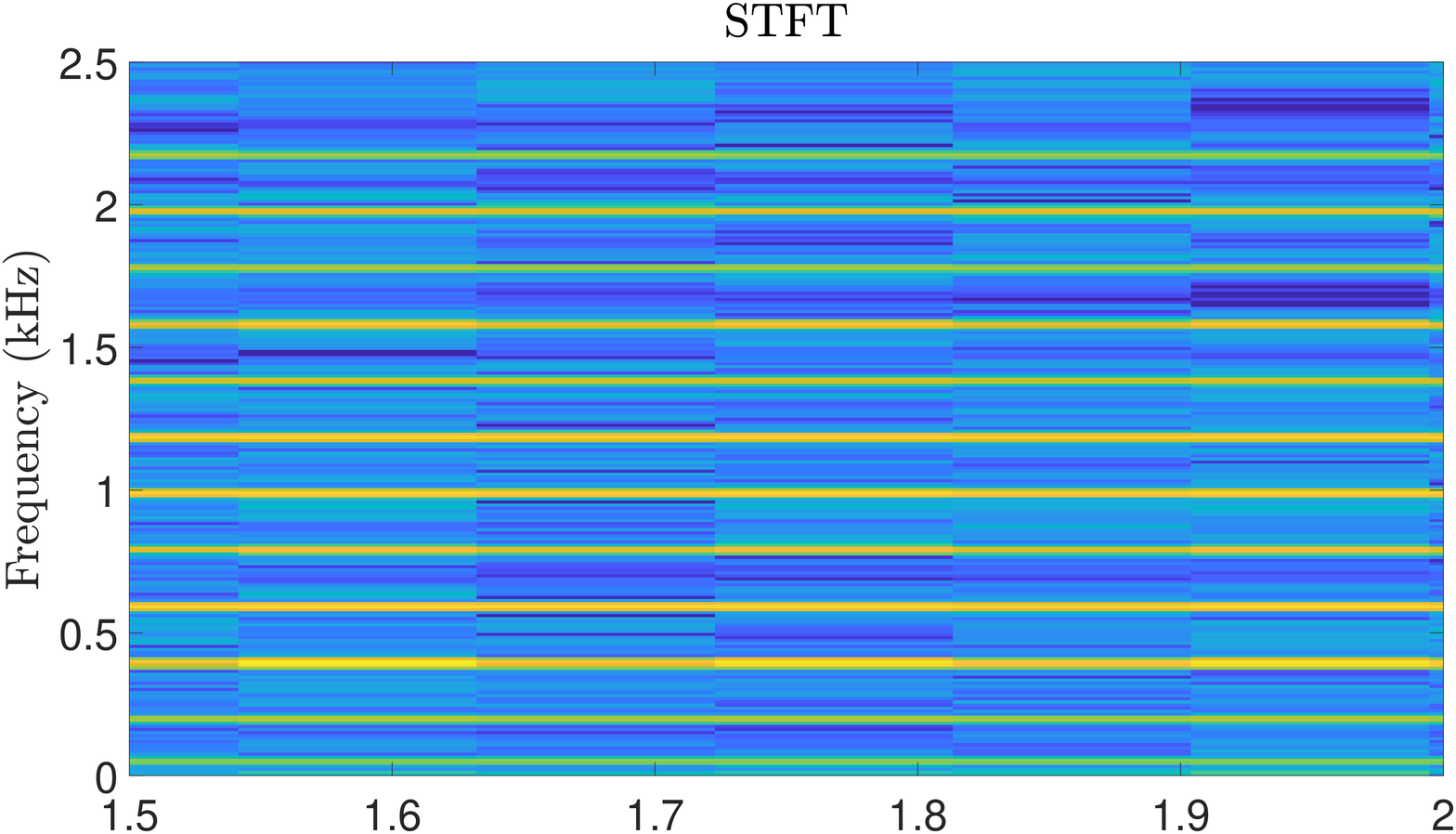}\\
\includegraphics[width=0.48\linewidth]{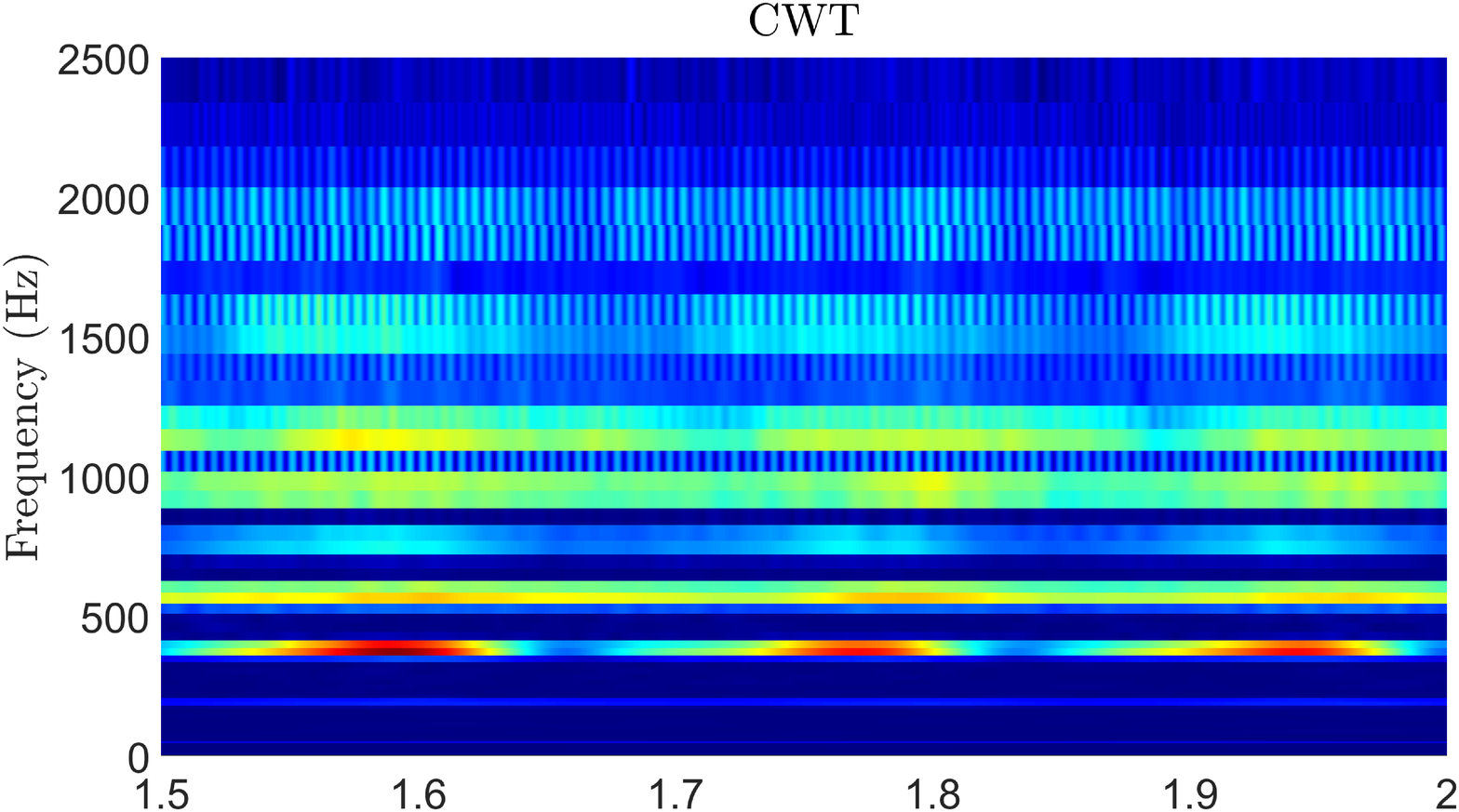}~\includegraphics[width=0.48\linewidth]{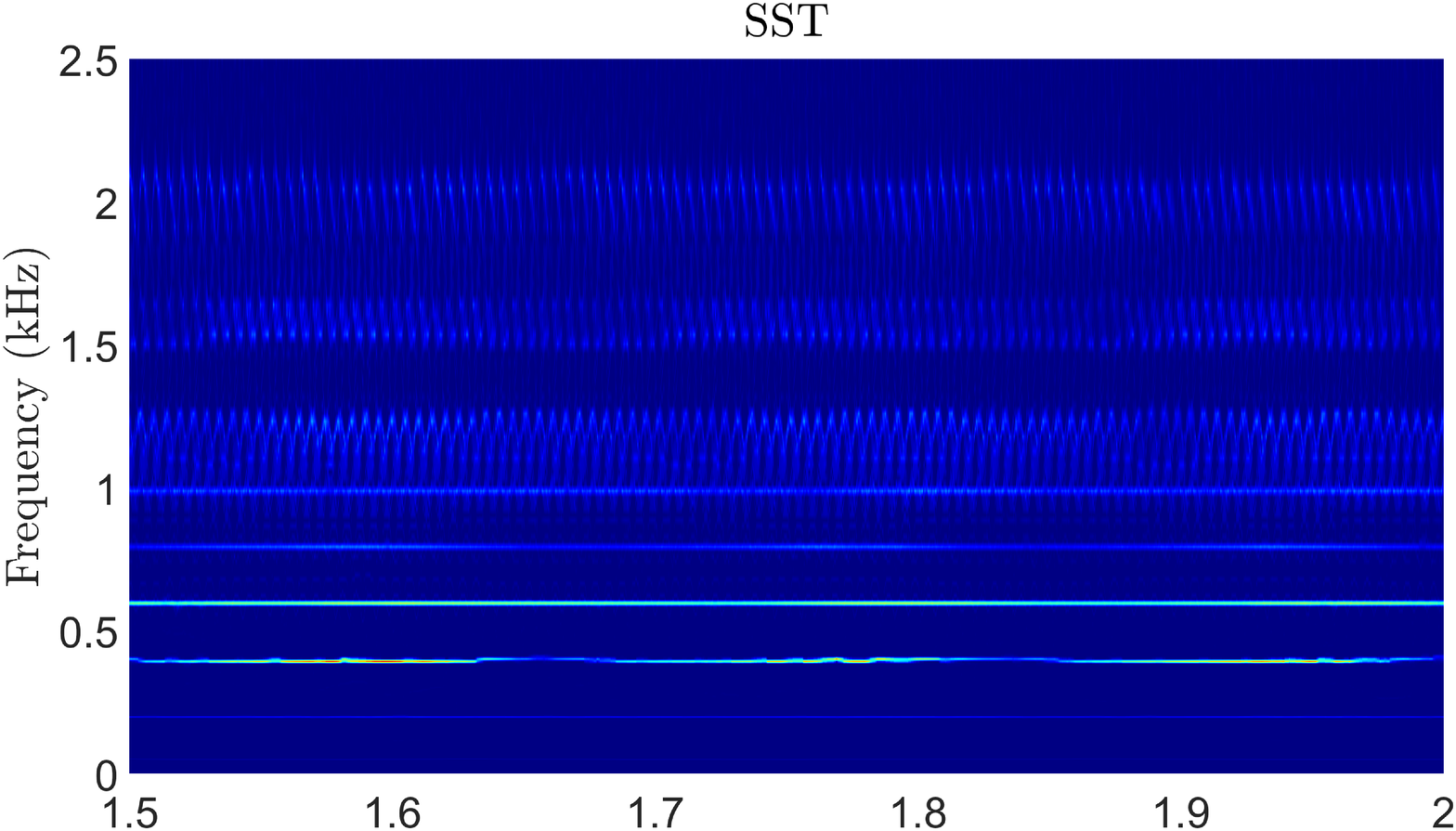}\\
\includegraphics[width=0.48\linewidth]{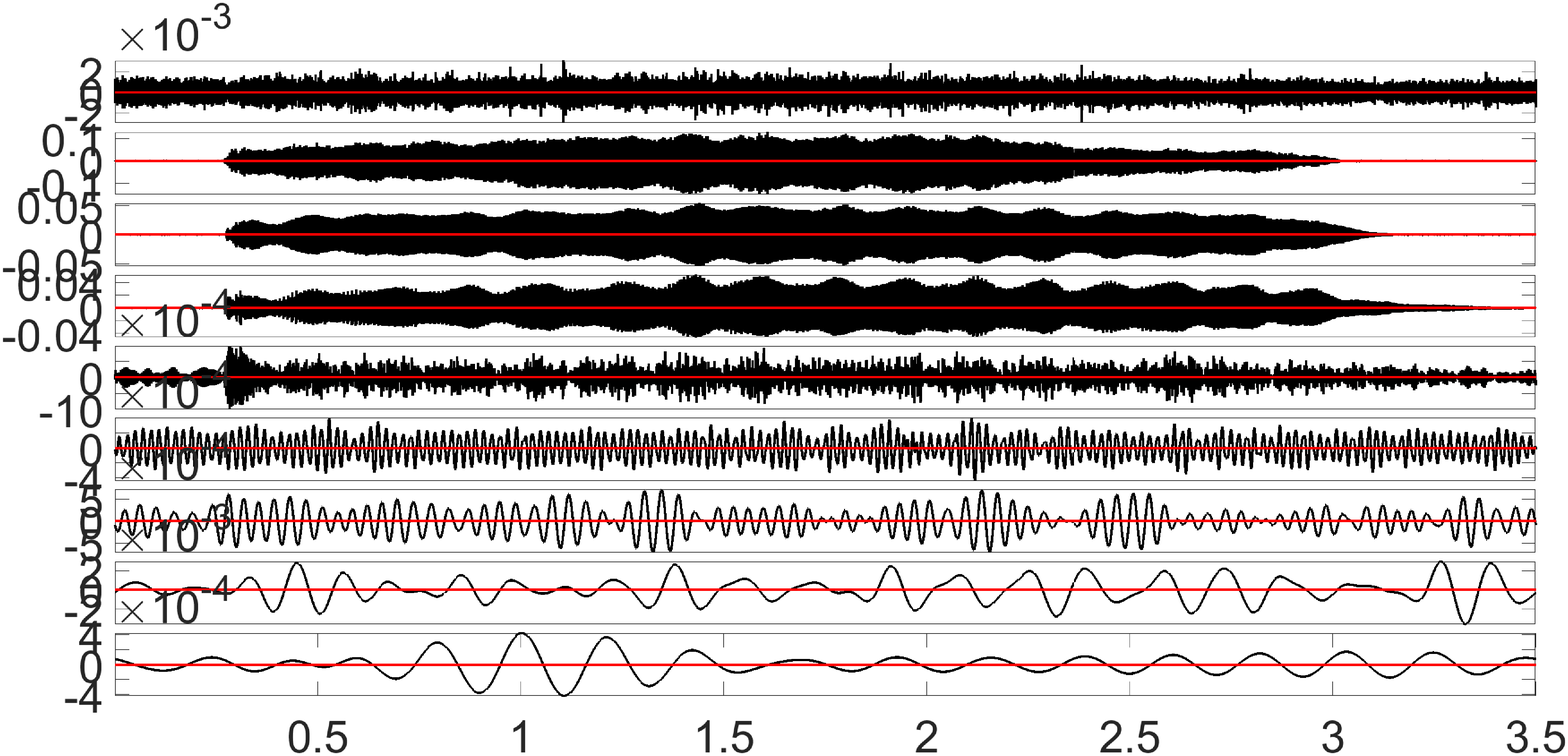}~\includegraphics[width=0.48\linewidth]{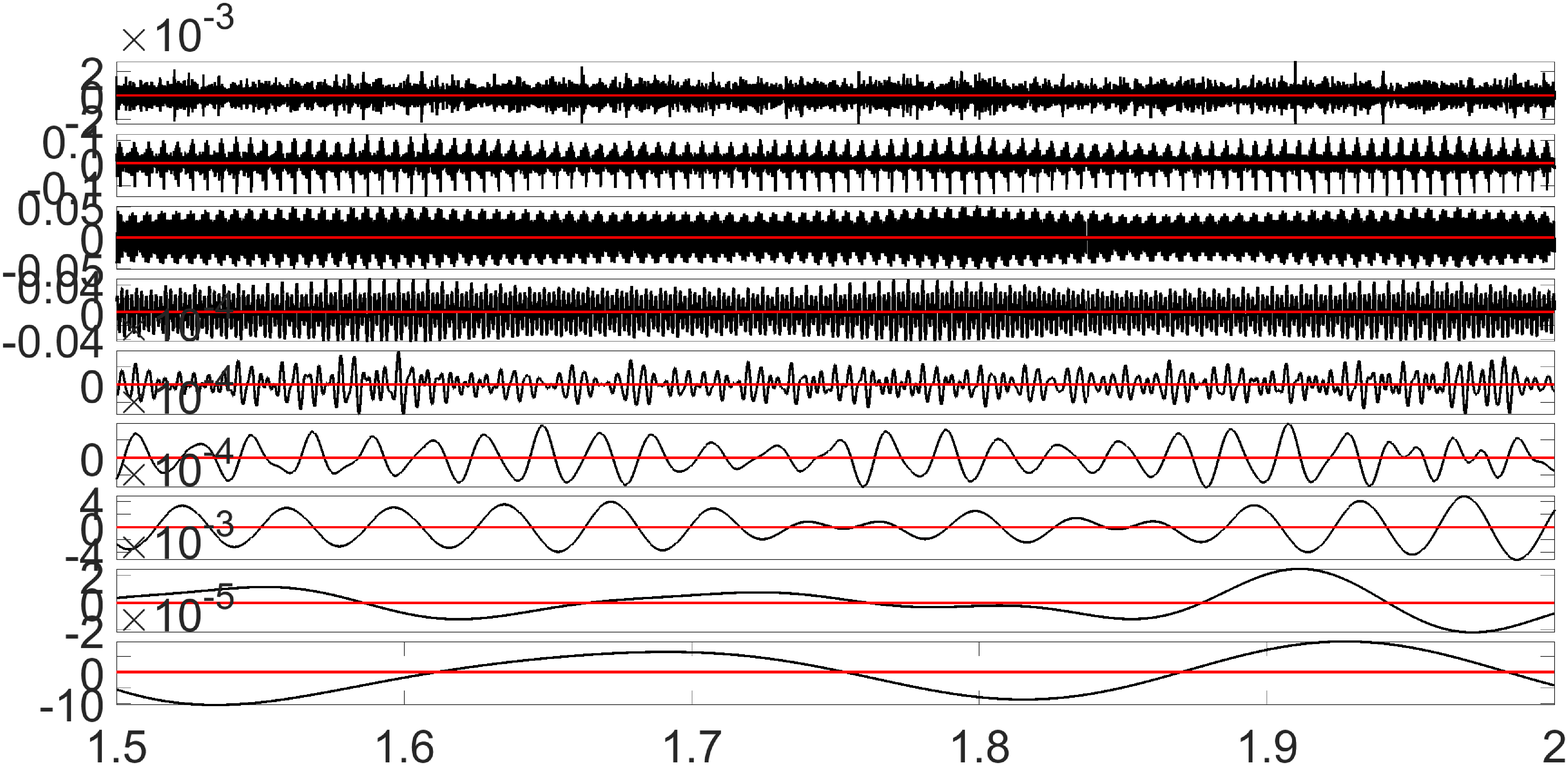}\\
\includegraphics[width=0.48\linewidth]{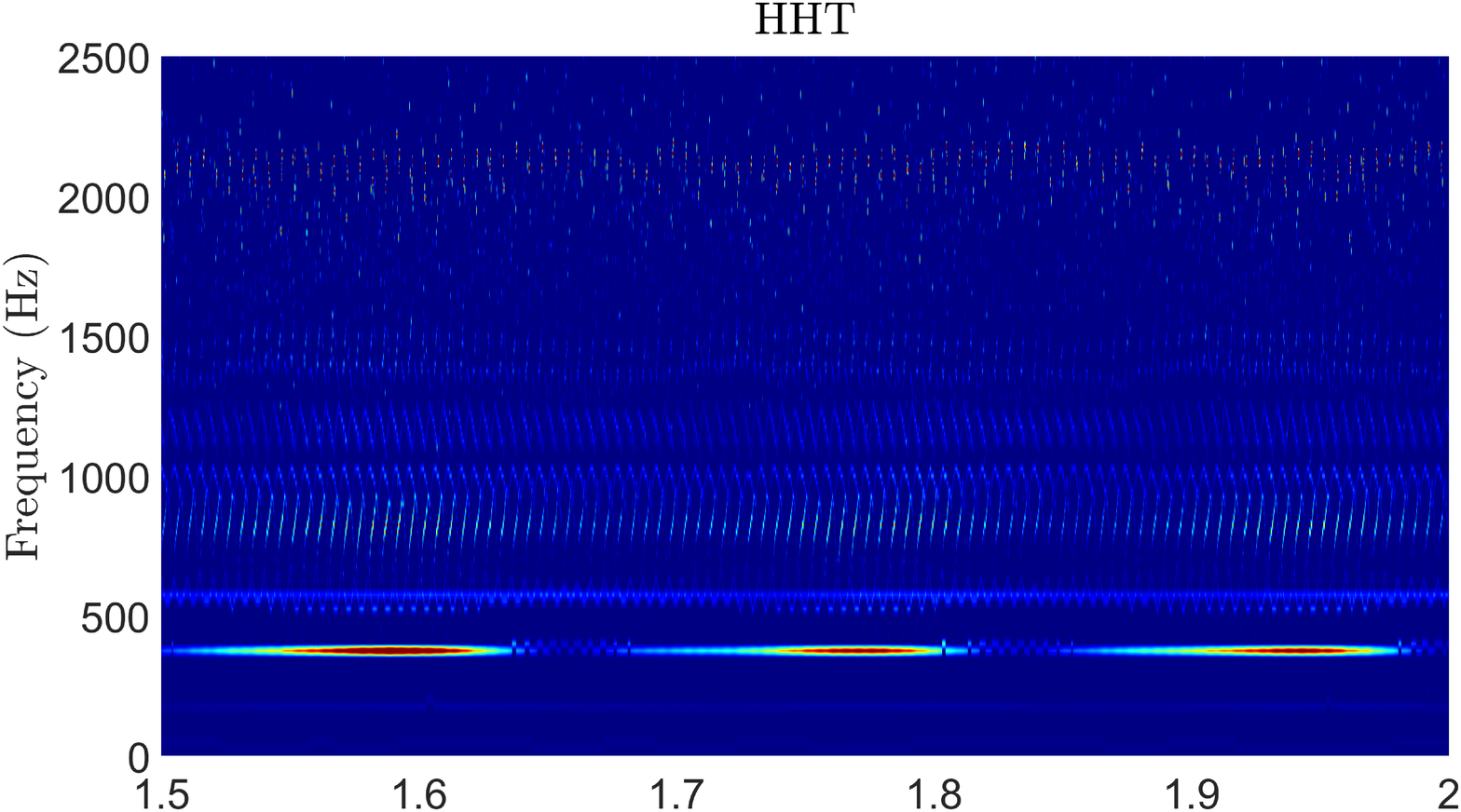}~\includegraphics[width=0.48\linewidth]{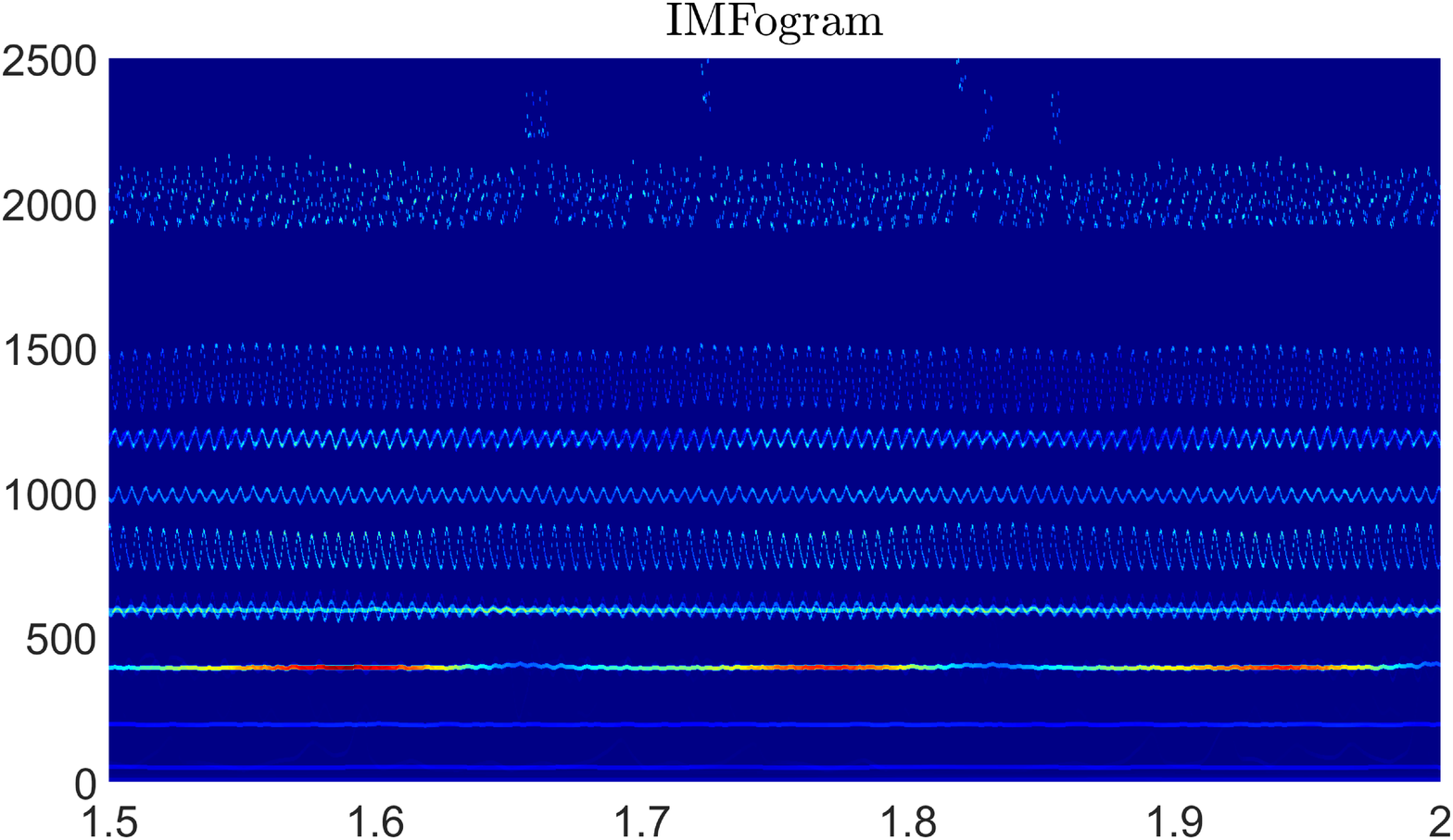}\\
\caption{Violin vibrato example: in the first row we plot the signal (left) and its STFT (right). In the second row we show the CWT (left) and SST representation (right). In the third row the signal IMF decomposition produced by FIF algorithm, left, and its zoomed in version in the time interval $[1.5,\, 2]$ seconds, right. Bottom row, HHT of the IMFs (left), and the IMFogram (right).}
\label{fig:vibrato}
\end{figure}

\subsection{Real life example -- Earth's ionospheric electron density}

In this last example, we consider the electron density (Ne), expressed in $\textrm{cm}^{-3}$, as measured in the topside ionosphere by the Langmuir Probes onboard one of the three spacecrafts constituting the Swarm\footnote{The Swarm dataset is the ``2 Hz Langmuir Probe Extended Dataset'' and it is provided by the European Space Agency (ESA) at \url{swarm-diss.eo.esa.int}  and at \url{http://vires.services.}} constellation of the European Space Agency \cite{friis2008swarm}. As an example, we consider the Ne samples measured at a 2 Hz rate by the Swarm Alpha satellite between 15:39 UT and 15:58 UT on 8 September 2017, covering the auroral and polar ionospheric sectors (magnetic latitudes below $55^{\circ}$ S) of the southern hemisphere, Figure \ref{fig:SWARMM-VIP} first row. The selected date is characterized by severe geomagnetic storm conditions triggered by Coronal Mass Ejections \cite{wu201904} that struck the Earth’s ionosphere, and caused the formation of ionospheric irregularities with largely varying spatial scales in the high-latitude ionosphere \cite{linty2018effects}. Swarm Alpha satellite flies in a quasi--polar orbit and, in the selected period, its speed was about 7.5 km/s and its altitude ranges between 440 km and 470 km.

\begin{figure}
\centering
\includegraphics[width=0.8\linewidth]{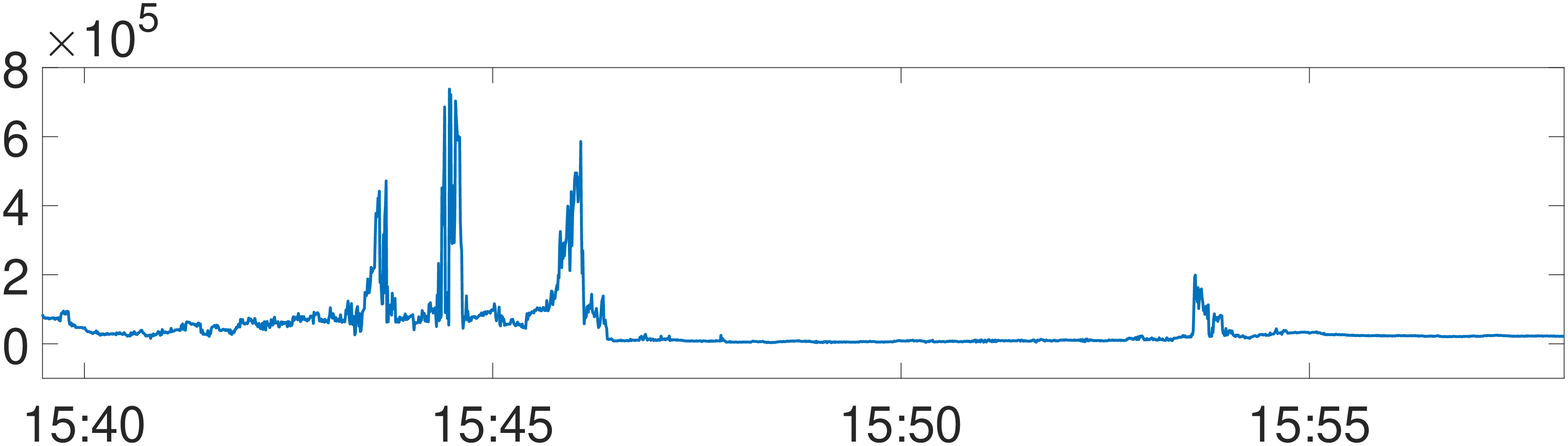}\\
\includegraphics[width=0.84\linewidth]{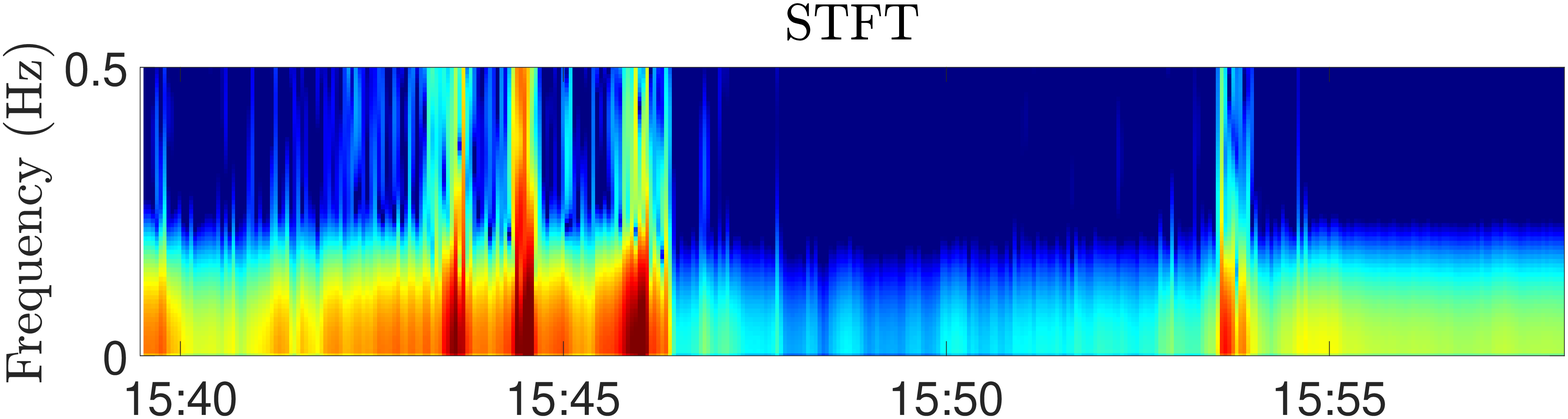}\\
\includegraphics[width=\linewidth]{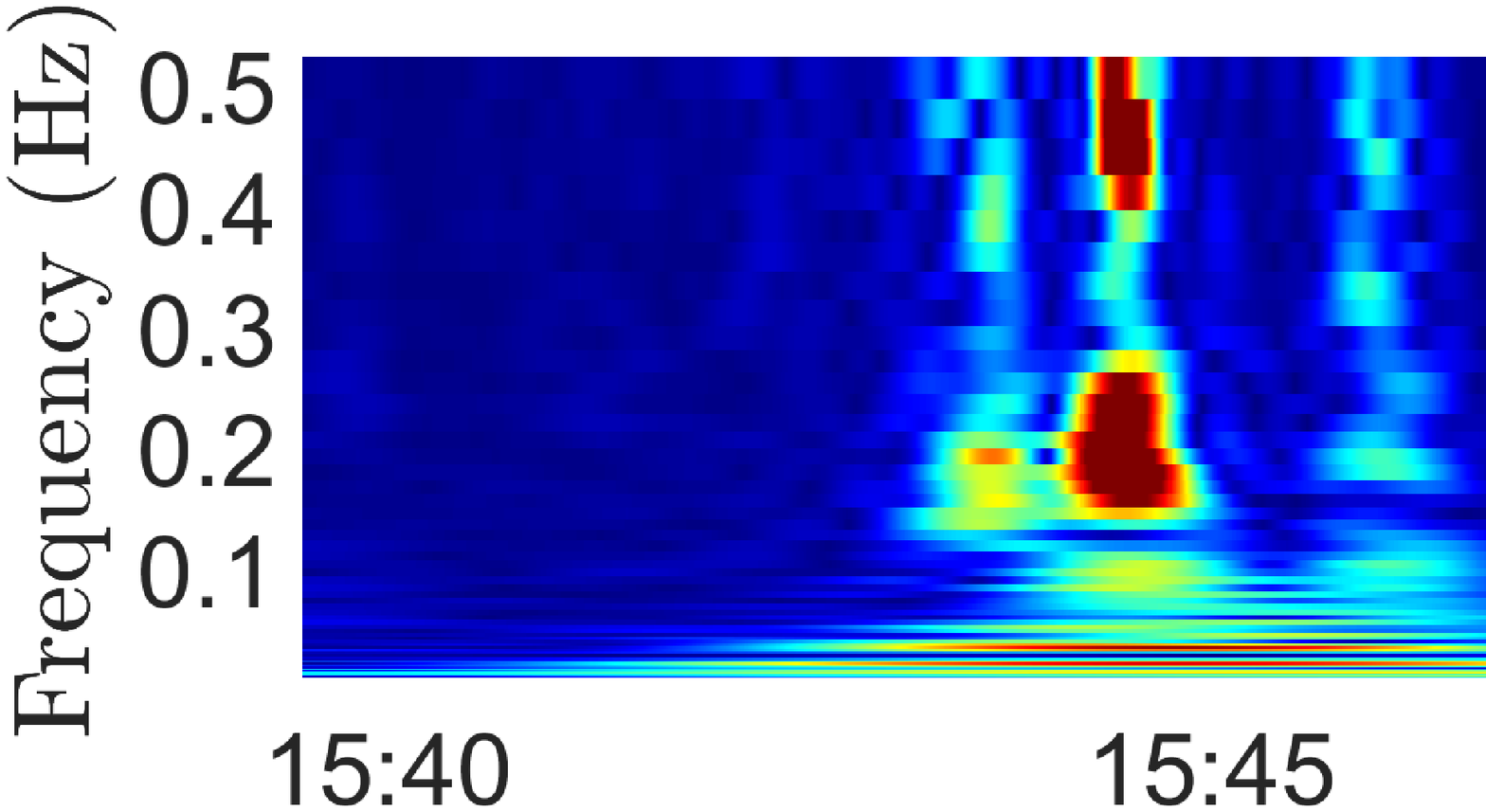}
\caption{Ionospheric electron density example: in the first row we plot the electron density (Ne) signal sampled by the ESA Swarm Alpha satellite between 15:39 UT and 15:58 UT on 8 September 2017. In the second and third row we report its STFT and CWT, respectively.}
\label{fig:SWARMM-VIP}
\end{figure}

From Figure \ref{fig:SWARMM-VIP} we can see that the STFT cannot achieve a high time-frequency localization accuracy, whereas the CWT can allow for a more accurate TFR.

\begin{figure}
\centering
\includegraphics[width=\linewidth]{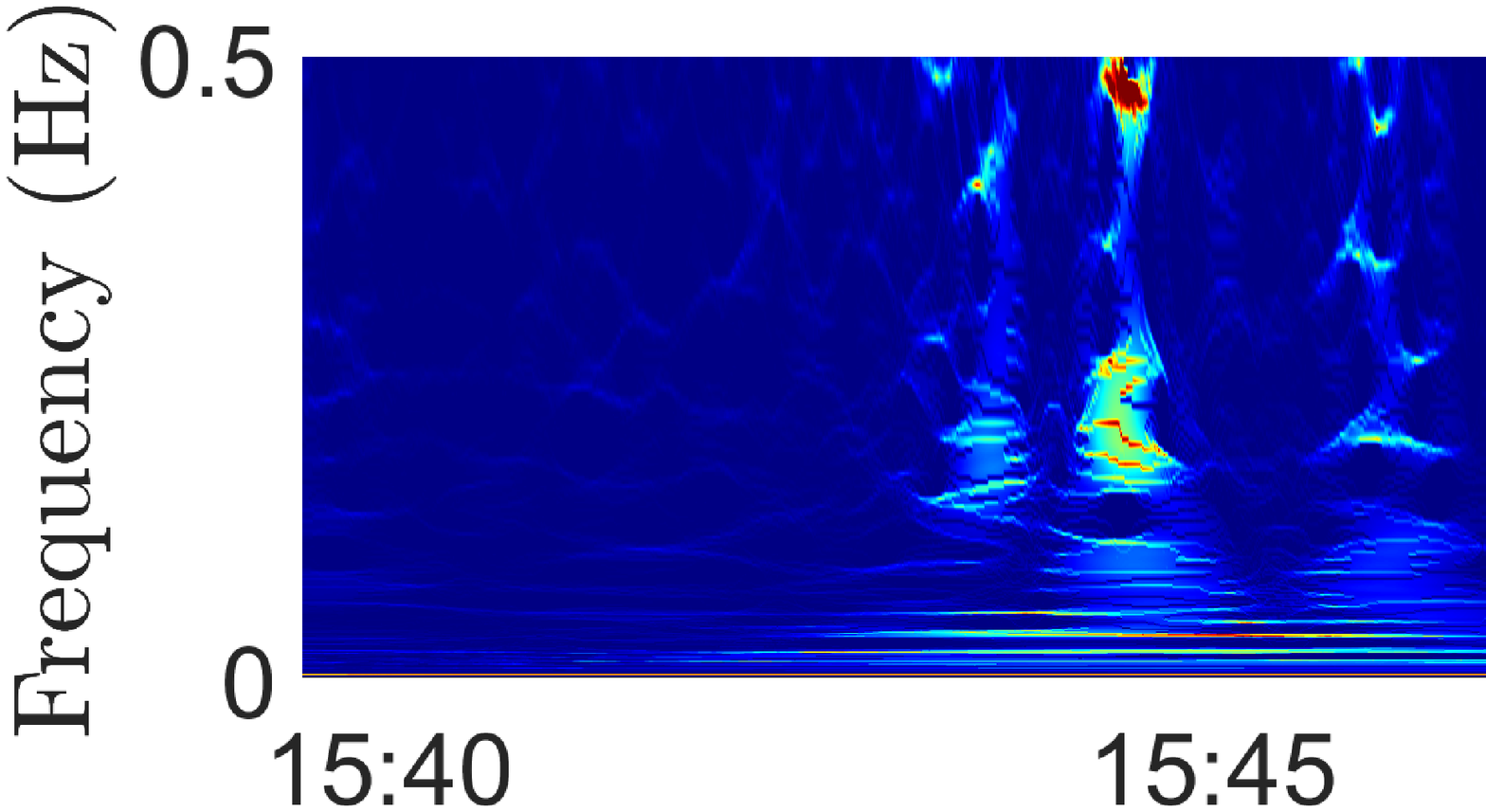}\\
\includegraphics[width=\linewidth]{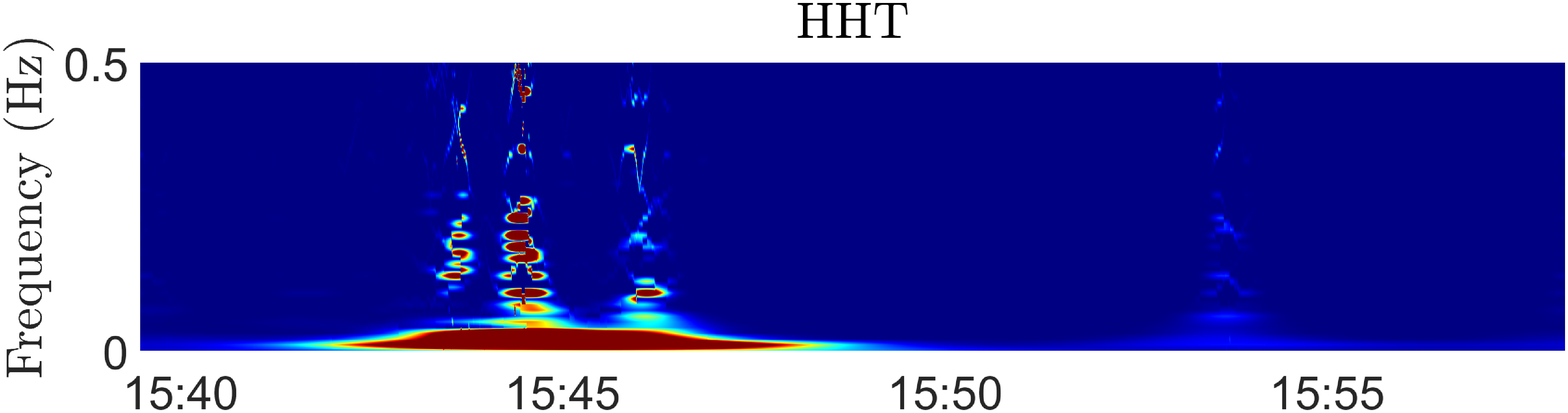}\\
\includegraphics[width=\linewidth]{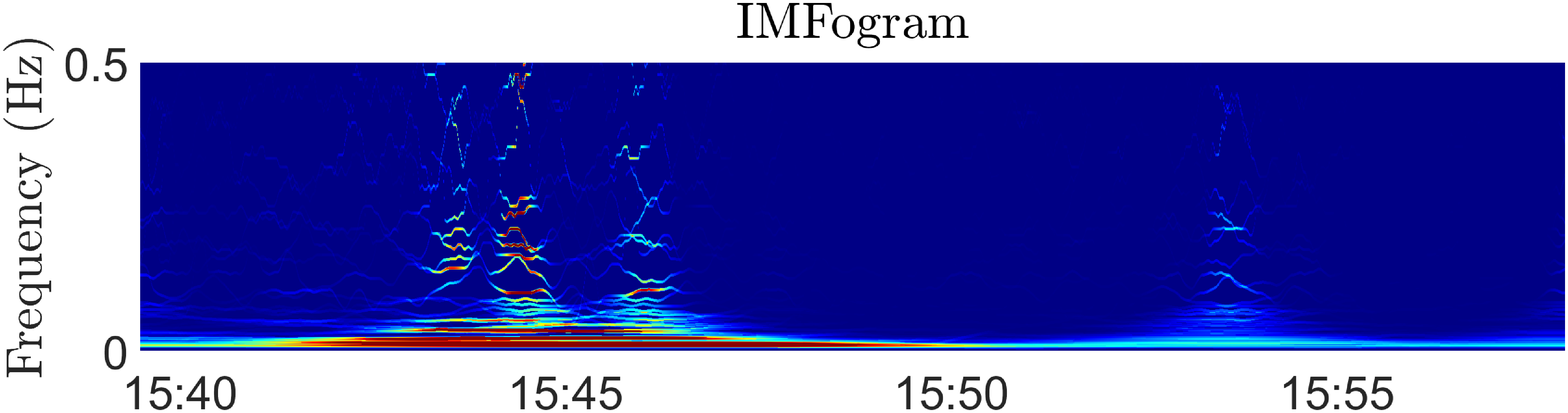}
\caption{Ionospheric electron density example: in the first row we show the SST representation of the sampled electron density. In the second and third row we present its HHT and IMFogram TFRs, respectively.}
\label{fig:SWARMM-VIP2}
\end{figure}

In Figure \ref{fig:SWARMM-VIP2}, first row, we show how the SST can improve, as expected theoretically \cite{daubechies1996nonlinear,daubechies2011synchrosqueezed}, the time and frequency localization obtained using the classical CWT. If we decompose the signal first into IMFs, then we can produce the HHT and IMFogram TFRs, which are plotted in the middle and bottom row of Figure \ref{fig:SWARMM-VIP2}. From these last two panels, we can see that the HHT has problems, at least in this example, capturing in an accurate way the energy and frequency of the IMFs, whereas the IMFogram proves to be really accurate, producing a plot which is even crisper and more focused than the SST.

\section{Conclusions}\label{sec:Conclusions}

In many real life applications the focus is on the time-frequency analysis of nonstationary signals. Many classical methods, like short time Fourier transform, continuous wavelet, and modern techniques, like reassignment, synchrosqueezing and alternative algorithms, prove to have limitations in producing crisp and focused time-frequency representations of signals. By leveraging on the groundbreaking idea of first decomposing a given signal into simple oscillatory components, a.k.a. intrinsic mode functions (IMFs), and then applying to each of them a time-frequency analysis, in this work we present the innovative IMFogram time-frequency representation method and its numerical analysis. In particular, after reviewing known theoretical properties of the so called Fast Iterative Filtering (FIF) method for the decomposition of a signal into IMFs, we propose new theoretical results regarding FIF and IMFogram. For FIF, we prove that the algorithm can preserve the ``$L_1$ Fourier Energy'' of the signal, which allows us to prove that FIF decomposition cannot contains ``unwanted oscillations'', meaning that all IMFs produced by FIF are meaningful. Furthermore, for the IMFogram algorithm, we prove that, under some hypothesis on the signal, the IMFogram time-frequency representation does converge to the corresponding spectrogram as the stopping criterion $\delta$ used in FIF to produce the IMFs is sent to 0.
We conclude this work with a few synthetic and real life numerical examples showing the performance of the proposed technique.

Even though many new insights on the decomposition into IMFs and their time-frequency representation are presented in this work, many problems still remain unsolved. First of all, iterative decomposition methods, like FIF--based and EMD--based techniques, and all other methods based on optimization, like sparse time-frequency representation, empirical wavelet transform, variational mode decomposition, just to name a few, cannot guarantee a uniqueness in the decomposition. How this can impact the associated time-frequency representation is yet to be studied.

In this work we propose a new definition of ``unwanted oscillations'' produced by a decomposition method. This definition is, to the best of our knowledge, the first of its kind. We assume that other definitions of what an unwanted oscillation is can be proposed in the context of signal decomposition and we plan to work in this direction in the future.

The IMFogram convergence to the spectrogram has been proved in this work for a nonstationary signal which is piece-wise stationary. This result can be extended to more general nonstationary signals. We plan to study this problem in a separate work.

Another interesting problem regards IMFogram time-frequency representation robustness to parameters tuning in the IMFogram itself, as well as in the FIF decomposition method. This is an open research direction. We plan to tackle it in a future work.

It is also important to point out that the main step of the IMFogram algorithm, i.e. \eqref{eq:IMFogram}, is meaningful because  the IMFs produced by IF and FIF methods are, based on numerical evidences we collected so far, in general quasi--orthogonal. However, this fact has not been proven rigorously yet. We plan to study this aspect in another work.

Finally, we point out that the proposed IMFogram time-frequency representation algorithm is based on the strong assumption that during each period of oscillation the IMFs have a constant instantaneous frequency. This is true for signals having interwave modulation of their frequencies. The very same assumption is made, implicitly, in all classical methods based on Fourier and wavelet transforms. However, many real life signals show what is called an intrawave modulation of their instantaneous frequencies. There is the need to develop a more general time-frequency representation able to handle also such intrawave modulations. This will be the subject of a future work.

\section*{Acknowledgments}

AC is a member of the INdAM Research group GNCS. WSL is partially supported by a Simon Foundation Collaboration Grant for Mathematicians. HZ is supported in part by NSF under grants DMS-1830225, and ONR N00014-18-1-2852.

\end{document}